\documentclass[10pt,a4paper]{amsart}
\usepackage[utf8]{inputenc}
\usepackage[english]{babel}
\usepackage{amsmath}
\usepackage{amsfonts}
\usepackage{amssymb}
\usepackage{hyperref}
\usepackage{tikz-cd} 

\theoremstyle{plain}
\newtheorem{theorem}{Theorem}[section]
\newtheorem{lemma}[theorem]{Lemma}
\newtheorem{prop}[theorem]{Proposition}
\newtheorem{coro}[theorem]{Corollary}

\newtheorem{question}{Question}
\newtheorem{conjecture}{Conjecture}

\theoremstyle{definition}
\newtheorem{defi}[theorem]{Definition}
\newtheorem{rema}[theorem]{Remark}
\newtheorem{example}[theorem]{Example}

\newcommand{\bR}{\mathbb{R}}
\newcommand{\bC}{\mathbb{C}}
\newcommand{\bQ}{\mathbb{Q}}
\newcommand{\bN}{\mathbb{N}}
\newcommand{\bZ}{\mathbb{Z}}
\newcommand{\bP}{\mathbb{P}}
\newcommand{\bK}{\mathbb{K}}

\newcommand{\cO}{\mathcal{O}}

\DeclareMathOperator{\Spec}{\mathrm{Spec}}
\DeclareMathOperator{\Rat}{\mathrm{Rat}}
\DeclareMathOperator{\ord}{\mathrm{ord}}
\DeclareMathOperator{\Div}{\mathrm{Div}}
\DeclareMathOperator{\divi}{\mathrm{div}}
\DeclareMathOperator{\Hom}{\mathrm{Hom}}

\DeclareMathOperator{\vol}{\mathrm{vol}}

\DeclareMathOperator{\Supp}{\mathrm{Supp}}
\DeclareMathOperator{\pam}{\widehat{\mu}^{\mathrm{asy}}_\mathrm{max}}
\DeclareMathOperator{\pmax}{\widehat{\mu}_\mathrm{max}}

\DeclareMathOperator{\pim}{\widehat{\mu}^{\mathrm{asy}}_\mathrm{min}}
\DeclareMathOperator{\pmin}{\widehat{\mu}_\mathrm{min}}

\DeclareMathOperator{\rk}{\mathrm{rk}}
\DeclareMathOperator{\Vect}{\mathrm{Vect}}

\DeclareMathOperator{\Essmin}{\zeta_{\mathrm{ess}}}
\DeclareMathOperator{\Essabs}{\zeta_{\mathrm{abs}}}
\DeclareMathOperator{\GL}{\mathrm{GL}}

\author{François Ballaÿ}
\title{Successive minima and asymptotic slopes in Arakelov Geometry}
\begin{document}

\maketitle

\begin{center}
\begin{small}
Beijing International Center for Mathematical Research, Peking University\\
5 Yi He Yuan Road, Beijing 100871, China\\
\url{francois.ballay@bicmr.pku.edu.cn}\\
\url{http://bicmr.pku.edu.cn/~ballay}
\end{small}
\end{center}

\vspace*{1cm}

\subsection*{Abstract} Let $X$ be a normal and geometrically integral projective variety over a global field $K$ and let $\overline{D}$ be an adelic Cartier divisor on $X$. We prove a conjecture of Chen, showing that the essential minimum $\zeta_{\mathrm{ess}}(\overline{D})$ of $\overline{D}$ equals its asymptotic maximal slope under mild positivity assumptions. As an application, we see that $\zeta_{\mathrm{ess}}(\overline{D})$ can be read on the Okounkov body of the underlying divisor $D$ via the Boucksom--Chen concave transform. This gives a new interpretation of Zhang's inequalities on successive minima and a criterion for equality generalizing to arbitrary projective varieties  a result of Burgos Gil, Philippon and Sombra concerning toric metrized divisors on toric varieties. When applied to a projective space $X = \bP_K^d$, our main result has several applications  to the study of successive minima of hermitian vector spaces. We obtain an absolute transference theorem with a linear  upper bound,  answering a question raised by Gaudron. We also give new comparisons between successive slopes  and absolute minima, extending results of Gaudron and Rémond.

\begin{flushleft}
\textbf{Keywords~:} Height, essential minimum, successive minima, adelic line bundles and divisors, Okounkov bodies, hermitian vector bundles, transference theorems.

\textbf{MSC Class~:} 14G40 (Primary) 11G50; 11H06 (Secondary) 
\end{flushleft}

\vspace*{2mm}

\subsection*{Acknowledgements} I am very grateful to Huayi Chen for many useful discussions and for communicating to me his article \cite{Cheniccm}. I am also indebted to José Ignacio Burgos Gil for drawing my attention to the potential links between the essential minimum and the concave transform function.  I warmly thank Pascal Autissier, Eric Gaudron and Gaël Rémond for their careful reading of the text and their very helpful remarks. Finally, I thank the organizers of the 2019 Intercity Seminar on Arakelov Geometry in Kyoto for giving me the opportunity to present this work.

\newpage

\section{Introduction}

 Let  $K$ be either a number field or $K = k(C_K)$ the field of functions of a regular projective curve $C_K$ defined over an arbitrary field $k$. Equivalently, $K$ is a finite extension of $K_0$, where $K_0$ denotes either $\bQ$ or $k(T)$. We let $\Sigma_K$ be the set of places of $K$ and we fix an algebraic closure $\overline{K}$ of $K$. Let $\pi \colon X \rightarrow \Spec K$ be a geometrically integral and normal projective variety on $\Spec K$ and let $d= \dim X$. We consider an adelic $\bR$-Cartier divisor $\overline{D}$  on $X$ with continuous metrics, and we denote by $h_{\overline{D}} \colon X(\overline{K}) \rightarrow \bR$ the height function associated to $\overline{D}$  (see section \ref{sectionadelicRdivisors} for details). The notion of adelic $\bR$-divisors is due to Moriwaki \cite{MoriwakiMAMS}, and  generalizes the one of adelic line bundles in the sense of Zhang \cite{Zhangadelic}.   The essential minimum of $\overline{D}$ is defined by
\[\Essmin(\overline{D}) = \sup_{Y \subsetneq X} \inf_{x \in X(\overline{K}) \setminus Y} h_{\overline{D}}(x),\]
where the supremum is taken over the Zariski-closed proper subschemes $Y \subsetneq X$. This invariant is of significant  importance in Diophantine geometry. The celebrated Bogomolov conjecture can be stated in terms of lower bounds for $\Essmin(\overline{D})$, which are required to be explicit in the effective version of the conjecture. Such bounds have been studied extensively and have applications to other problems in Diophantine geometry and Arakelov geometry. A striking example is given by Zhang's inequalities on successive minima \cite[Theorem 5.2]{Zhangplav} (see Theorem \ref{thmZhangintro} below),  which also play an important role in equidistribution problems. Indeed, classical equidistribution theorems (including the ones of Szpiro--Ullmo--Zhang \cite{SUZ}, Favre--Rivera-Letelier \cite{FavreRivera}, Chambert-Loir \cite{ChL}, Baker--Rumely \cite{BakerRumely},  Yuan \cite{Yuanbig} and Berman--Boucksom \cite{BermanBoucksom}) are applicable only when equality holds in Zhang's theorem (see \cite{BPRS},\cite{BPSmin} and subsection \ref{paragZhangintro} below). The invariant $\Essmin(\overline{D})$ is also related to important problems in geometry of numbers. Indeed, when $X = \bP^d_K$ is a projective space the essential minimum encodes crucial arithmetic information of euclidean lattices, or more generally hermitian $K$-vector spaces (see \cite{GR13, GRSiegel} and subsection \ref{paragGNintro} of this paper). 

In \cite{BPSmin}, Burgos Gil, Philippon and Sombra proposed an innovative study of the essential minimum for toric metrized divisors on toric varieties based on convex analysis, and showed that in this case  $\Essmin(\overline{D})$ coincides with the maximum of a concave function defined on the (geometric) Okounkov body of $D$. In a joint work with Rivera-Letelier \cite{BPRS}, they applied this result to investigate new situations in which equidistribution phenomena occur (in the toric setting).  The results of \cite{BPSmin} and \cite{BPRS} give a completely new understanding of the essential minimum when $\overline{D}$ is a toric metrized divisor. However, the toric assumption on the metrics is rather restrictive (even in the case $X = \bP_K^d$), and confers a very specific behaviour to the essential minimum. To give a better understanding of the invariant $\Essmin(\overline{D})$ in the general case is an important topic of research, in which many interesting questions remain unanswered.

In \cite{Chenpolygone}, Chen introduced an invariant measuring the asymptotic minimal size of global sections of a hermitian line bundle. For any integer $n \geq 1$, the $K$-vector space $V_n:= H^0(X,nD)$ is equipped with a supremum norm $\|.\|_{v,\sup}$ for each place $v$ of $K$. Let $v_0 \in \Sigma_{K_0}$ be any place of $K_0$. We denote by $ \lambda_{\max,n}(\overline{D})$ the supremum of the real numbers $t$ for which there exists a non-zero function $\phi \in V_n$ such that for every place $v \in \Sigma_K$, 
 \[\|\phi\|_{v,\sup} \leq \left\{\begin{tabular}{ll}
$e^{-t}$ & if $ v | v_0$,\\
$1$ & otherwise.
 \end{tabular}\right.\] 
When $D$ is big, the sequence $(\lambda_{\max, n}(\overline{D})/n)_{n \in \bN}$ converges in $\bR$. Following \cite{Chenpolygone}, we call its limit the \textit{asymptotic maximal slope} of $\overline{D}$ and denote it by $\pam(\overline{D})$. The following conjecture of Chen relates this invariant to the essential minimum.
\begin{conjecture}[Chen]\label{conjChen} Assume that $\overline{D}$ is semi-positive and that the underlying divisor $D$ is big. Then $\Essmin(\overline{D}) =  \pam(\overline{D})$.
\end{conjecture}
 An explicit statement of this conjecture can be found in \cite[Conjecture 4.1]{Cheniccm}, \cite[section 5.3]{Chencompare}. Although it appeared recently in the literature, Conjecture \ref{conjChen} has been known to experts in Arakelov geometry  for some time and the potential relations between the essential minimum and the asymptotic slope were already discussed in the work of Chen on the differentiability of the arithmetic volume \cite[section 5.2]{Chendiff} (together with applications to equidistribution). In the particular case of a toric metrized divisor on a toric variety, Conjecture \ref{conjChen} is  a consequence of the work of Burgos Gil, Philippon and Sombra \cite{BPSmin} on the essential minimum (see \cite[Corollary 3.10]{BPSmin} and the discussion preceding Theorem \ref{thmmainintro} below). The behaviour of $\Essmin(\overline{D})$ is more subtle in the non-toric setting, and the conjecture remains open in general. The relevance of Conjecture \ref{conjChen} is to suggest a completely new approach to study the essential minimum, allowing one to compute $\Essmin(\overline{D})$ only in terms of the arithmetic of the graded linear series $V_\bullet := \bigoplus_{n \in \bN} V_n$. It also has deep applications in geometry of numbers, as we will see in subsection \ref{paragGNintro}. 
 
 \vspace{4mm}
 
The main achievement of this paper is a proof of Conjecture \ref{conjChen} (see Theorem \ref{thmmain2intro} below). Our approach starts  with the equivalence
\begin{equation}\label{equivbigpam}
\pam(\overline{D}) > 0 \Longleftrightarrow \overline{D} \text{ is big}
\end{equation}
due to Chen (see for example \cite[Proposition 3.11]{Chenfujita}, or Proposition \ref{proppamsup} in this paper). Here the condition that $\overline{D}$ is big means that it has a positive arithmetic volume $\widehat{\vol}(\overline{D}) > 0$ (see subsection \ref{sectionvolume} for details). Using \eqref{equivbigpam} and rescaling norms, we shall see that Conjecture \ref{conjChen} is equivalent to the following theorem, which is the key result of this paper. We say that $\overline{D}$ is pseudo-effective if $\overline{D} + \overline{B}$ is big for any big adelic $\bR$-Cartier divisor $\overline{B}$. 
\begin{theorem}[Theorem \ref{thmmain}]\label{thmmainintro} Assume that $\overline{D}$ is semi-positive and that $D$ is big. Then we have the equivalence
\[\Essmin(\overline{D}) \geq 0 \Longleftrightarrow \overline{D} \textit{ is pseudo-effective}.\]
\end{theorem}
 When $D$ is big and $\overline{D}$ is pseudo-effective, it is not hard to see that $\Essmin(\overline{D})\geq 0$. The other implication in the theorem is more challenging. To outline the strategy of the proof, we consider the following simplified setting~: $K=k(C_K)$ is a function field and there exists a normal and proper model $\pi_{\mathcal{X}} \colon \mathcal{X} \rightarrow C_K$ of $X$, together with a relatively nef $\bR$-Cartier divisor $\mathcal{D}$ on $\mathcal{X}$, such that for any $v \in \Sigma_K$ the metric on $D$ at the place $v$ is induced by $\mathcal{D}$. We have a Cartesian diagram
\begin{equation*}
\begin{tikzcd}
X \arrow{r} \arrow{d}{\pi} & \mathcal{X} \arrow{d}{\pi_{\mathcal{X}}}  \\
\Spec K  \arrow{r}& C_K. 
\end{tikzcd}
\end{equation*} 
 In this case, we can interpret the height function as an intersection number. More precisely, for any point $x \in X(\overline{K})$ with residue field $K(x)$, we have
\[h_{\overline{D}}(x) = \frac{\mathcal{D} \cdot C_x}{[K(x) : K_0]},\]
where $C_x$ is the Zariski-closure of $x$ in $\mathcal{X}$. Note that $C_x \subset \mathcal{X}$ is an integral curve and $\pi_{\mathcal{X}}(C_x) = C_K$.  The key ingredient of the proof is a deep theorem of Boucksom--Demailly--P\u{a}un--Peternell \cite{BDPP}, which states that the pseudo-effective cone of $\mathcal{X}$ is dual to the cone of movable curves.  If $\mathcal{D}$ is not pseudo-effective, this result implies the existence of a family of integral curves $(C_t)_{t \in T}$ covering a dense subset of $\mathcal{X}$ and such that $\mathcal{D}\cdot C_t < 0$ for any $t \in T$. Since  $\mathcal{D}$ is relatively nef by assumption, all  the curves $C_t$ are horizontal, i.e.\ $\pi_{\mathcal{X}}(C_t) = C_K$ for every $t \in T$.  If we denote by $\eta_t \in X(\overline{K})$ the generic point of $C_t$ for every $t \in T$, the set 
\[\Lambda_T = \{\eta_t \ | \ t \in T\} \subset X(\overline{K})\]
is dense and consists of points with negative height, hence $\Essmin(\overline{D}) \leq 0$. By contraposition, this proves the implication
\[\Essmin(\overline{D}) > 0 \implies \overline{D} \text{ is pseudo-effective}.\]
If $\Essmin(\overline{D}) = 0$, the above implication together with a limit argument shows that $\overline{D}$ is also pseudo-effective.  When $K$ is a number field, we use arithmetic intersection theory to adapt the above argument. We first reduce the problem to a fixed model and work with an arithmetic variety on $\Spec \cO_K$, and then apply an  arithmetic analogue of Boucksom--Demailly--Pa\u{u}n--Peternell's theorem due to Ikoma \cite{Ikoma15}. To do so we also use an arithmetic Bertini-type theorem due to Moriwaki \cite{Moriwaki94}.   

 One can interpret Theorem \ref{thmmainintro} as a numerical criterion for arithmetic pseudo-effectiveness. As mentioned before, it is equivalent to Conjecture \ref{conjChen} (see Remark \ref{remapamcursor}). We will prove the following theorem, which also gives an analogue of Conjecture \ref{conjChen} relating the absolute minimum $\zeta_{\mathrm{abs}}(\overline{D}) := \inf_{x \in X(\overline{K})}h_{\overline{D}}(x)$ and the \textit{asymptotic minimal slope} $\pim(\overline{D})$ (see Definition \ref{defiasyslopes}). The latter relies on an arithmetic Nakai--Moishezon's criterion due to Zhang \cite{Zhangplav}, generalized by Chen and Moriwaki \cite{CMadeliccurves}.
\begin{theorem}[Theorem \ref{thmpamessmin}]\label{thmmain2intro}
Assume that $\overline{D}$ is semi-positive and that $D$ is big. Then 
\[\Essmin(\overline{D}) = \pam(\overline{D}).\]
If moreover $D$ is a semi-ample $\bQ$-divisor, then
\[\Essabs(\overline{D}) = \pim(\overline{D}).\]
\end{theorem}

We present applications of Theorem \ref{thmmain2intro} in two directions. In the spirit of \cite{BPSmin}, we first give an interpretation of the essential minimum through convex analysis, and study some consequences for  Zhang's theorem on minima. Secondly we apply Theorem \ref{thmmain2intro} to projective spaces to study successive minima of hermitian vector spaces.  

\subsection{Essential minimum and Okounkov body} Assume that $D$ is big. In \cite{BPSmin}, Burgos Gil, Philippon and Sombra proposed a systematic study of the essential minimum when $\overline{D}$ is a toric metrized divisor on a toric variety. The main theorem of \cite{BPSmin} (Theorem A) relates $\Essmin(\overline{D})$ to the maximum of the \textit{roof function} $\vartheta_{\overline{D}} \colon \Delta \rightarrow \bR$, which is a concave function encoding arithmetic data of $\overline{D}$ and defined on the Okounkov body $\Delta$ of $D$. The latter is a convex body in $\bR^d$ and was introduced independently in the seminal papers of Lazarsfeld-Musta\c{t}\u{a} \cite{LazMus} and Kaveh-Khovanskii \cite{KavehKhovanskii}. In the toric case, the Okounkov body is an intrinsic data of $D$. In general, the definition is more complicated  and depends on the choice of a system of parameters $\mathbf{z}$  on $X_{\overline{K}}$ (see section \ref{sectionOkounkov} for details). Using Theorem \ref{thmmain2intro}, we will however be able to generalize \cite[Theorem A]{BPSmin} to the case where $X$ is not necessarily toric. Let $\Delta_{\mathbf{z}}(D)\subset \bR^d$ be the Okounkov body of the divisor $D \in \Div(X)_\bR$ defined with respect to a given system $\mathbf{z}$. In \cite{BoucksomChen}, Boucksom and Chen introduced a function $G_{\overline{D},\mathbf{z}}\colon \Delta_{\mathbf{z}}(D) \rightarrow \bR \cup \{-\infty\}$, called the concave transform of $\overline{D}$.  The function $G_{\overline{D},\mathbf{z}}$ generalizes to the non-toric case the roof function $\vartheta_{\overline{D}}$ of Burgos Gil--Philippon--Sombra (see \cite[Introduction]{BMPS}). It turns out that the maximum of the concave transform coincides with the asymptotic maximal slope. As a consequence of Theorem \ref{thmmain2intro}, we have the following.
\begin{coro}[Proposition \ref{propEssminmaxcf}]\label{coroEssminmaxcfintro} If $D$ is big, then 
$\Essmin(\overline{D}) \geq \max_{\alpha \in \Delta_{\mathbf{z}}(\overline{D})}G_{\overline{D},\mathbf{z}}(\alpha)$,
with equality if $\overline{D}$ is semi-positive.
\end{coro}
 When $\overline{D}$ is semi-positive and $D$ is big, this corollary permits to read the essential minimum directly on the Okounkov body of $D$. In addition to the theoretical interest of this result, its relevance lies in the fact that several invariants associated to $\overline{D}$ can be easily computed using the concave transform \cite[sections 2.4 and 3.1]{BoucksomChen}. For example, we have the following formula \cite[Theorem 3.1]{BoucksomChen} for the $\chi$-volume $\widehat{\vol}_{\chi}(\overline{D})$ of $\overline{D}$ (under mild positivity assumptions on $\overline{D}$, see subsection \ref{sectionconctransfvol} for details)~:
 \begin{equation}\label{eqvolchiintro}
 \widehat{\vol}_{\chi}(\overline{D}) = [K:K_0](d+1)!\int_{\Delta_{\mathbf{z}}(\overline{D})} G_{\overline{D},\mathbf{z}}d\lambda.
 \end{equation}
 We present an application of Corollary \ref{coroEssminmaxcfintro} to Zhang's theorem in the next paragraph.

\subsection{On Zhang's theorem on minima}\label{paragZhangintro}  The essential and absolute minima are part of a series of invariants associated to $\overline{D}$, originally introduced by Zhang \cite{Zhangplav}. For all $\lambda \in \bR$, we denote by $X_{\overline{D}}(\lambda)$ the Zariski-closure in $X$ of the set 
\[\{x \in X(\overline{K}) \ | \ h_{\overline{D}}(x) \leq \lambda\}.\]
For every integer $i \in \{1, \ldots, d +1\}$,  the $i$\textit{-th Zhang minimum} $\zeta_i(\overline{D})$ is defined by
\[\zeta_i(\overline{D}) = \inf \{\lambda \in \bR \ | \ \dim X_{\overline{D}}(\lambda) \geq i-1 \}.\]
Note that  $\zeta_{d+1}(\overline{D}) = \Essmin(\overline{D})$ and $\Essabs(\overline{D}) = \zeta_1(\overline{D})$. We denote by  $h_{\overline{D}}(X)$ the height of $X$ with respect to $\overline{D}$ (see \cite[section 2.5]{BPS14} for a definition).  A celebrated theorem of Zhang \cite[Theorem 5.2]{Zhangplav} relates this invariant to the successive minima $\zeta_i(\overline{D})$.  It was generalized by Gubler \cite[Proposition 5.10]{Gublerequi} for global fields with weaker positivity assumptions on $\overline{D}$.
\begin{theorem}[Zhang, Gubler]\label{thmZhangintro} Assume that $\overline{D} \in \widehat{\Div}(X)_\bQ$ is semi-positive and that the underlying divisor $D$ is big and semi-ample. Then 
\[(d+1)\Essmin(\overline{D}) \geq \frac{h_{\overline{D}}(X) }{D^d} \geq \sum_{i=1}^{d+1} \zeta_i(\overline{D}).\]
\end{theorem}
 The first inequality (referred to as the ``fundamental inequality" in the articles of Gubler \cite{GublerBogomolov, Gublerequi})  is of particular interest in equidistribution problems and in proofs of the Bogomolov Conjecture. Except from the recent results of \cite{BPRS} concerning the toric case, all the equidistribution theorems in the literature apply only when the latter is an equality, i.e.\ when
 \begin{equation}\label{equalityZhangintro}
\Essmin(\overline{D}) =\frac{h_{\overline{D}}(X) }{ (d+1)D^d}.
 \end{equation}
It is therefore of particular interest to characterize when this equality occurs. We refer the reader to the introductions of the papers \cite{BPRS, BPSmin} for a complete discussion and more details on this problem. When $X$ is a toric variety and $\overline{D}$ is a toric metrized divisor,  Burgos Gil, Philippon and Sombra   show that \eqref{equalityZhangintro} holds if and only if the roof function $\vartheta_{\overline{D}}$ is constant \cite[Corollary E]{BPSmin}. Combining Corollary \ref{coroEssminmaxcfintro} with \eqref{eqvolchiintro}, we are able to generalize this result to the non-toric setting. We also obtain a straightforward proof of the first inequality in Theorem \ref{thmZhangintro}, valid for $\bR$-divisors and without semi-ampleness assumption.
\begin{theorem}[Corollary \ref{coroineqZhangheight}]\label{thmineqZhangintro} If $\overline{D}$ is semi-positive and $D$ is big, then
\[\Essmin(\overline{D}) \geq \frac{h_{\overline{D}}(X) }{(d+1)D^{d}},\]
with  equality if and only if the following equivalent conditions are satisfied~:
\begin{enumerate}
\item $G_{\overline{D},\mathbf{z}}$ is constant for any $\mathbf{z}$;
\item\label{condeqZintro2} the sequence $(\lambda_{\max,n}(\overline{D})/n)_{n \geq 1}$ converges to $\frac{h_{\overline{D}}(X) }{(d+1)D^{d}}$.
\end{enumerate}
In that case, $G_{\overline{D},\mathbf{z}}(\alpha) = \Essmin(\overline{D})$ for any $\mathbf{z}$ and any $\alpha \in \Delta_{\mathbf{z}}(\overline{D})$.
\end{theorem}
  As in \cite{BPSmin}, we also prove analogues of this theorem when $h_{\overline{D}}(X)$ is replaced by the $\chi$-volume $\widehat{\vol}_{\chi}(\overline{D})$ or the arithmetic volume $\widehat{\vol}(\overline{D})$ (see Theorem \ref{thmineqZhang}). Criterion \eqref{condeqZintro2} in Theorem \ref{thmineqZhangintro} has a nice interpretation through Gaudron's slope theory for adelic vector spaces \cite{Gaudronpentes}, and can be thought of as an ``asymptotic semi-stability" condition (Remark \ref{remacritasystable}). We will study this approach further in section \ref{sectionGN}, restricting our attention to projective spaces $X = \bP_K^d$. 
  
\subsection{Applications to geometry of numbers}\label{paragGNintro} In section \ref{sectionGN} we apply Theorem \ref{thmmain2intro} to the study of $K$-hermitian vector spaces $\overline{E} = (E, (\|.\|_v)_{v \in \Sigma_K})$, which generalize  the notion of euclidean lattices to global fields (see section \ref{sectionhermitian} for definitions). To clarify the exposition we assume that $K$ is a number field in the end of this introduction. In  geometry of numbers, the central objects of study associated to a hermitian $K$-vector space $\overline{E}$ are  its \textit{successive minima}. Various definitions have been introduced by different authors, such as Bombieri--Vaaler \cite{BombieriVaaler}, Roy--Thunder \cite{RoyThunder}, Zhang \cite{Zhangplav}, etc.  We shall focus on the two following definitions. Let  $h_{\overline{E}} \colon E \otimes_K \overline{\bQ} \rightarrow \bR$ be
the logarithmic height function associated to $\overline{E}$ and let $d= \dim E$. For any algebraic extension $K'$ of $K$ and any $\lambda \in \bR$, we consider the ball 
 \[E(\lambda,K') = \{ s \in E \otimes_K K' \ | \ h_{\overline{E}}(s) \leq \lambda\}. \]
 For $i \in \{1,\ldots,d \}$, we define
 the \textit{Roy-Thunder $i$-th minimum} with respect to $K'$ by 
 \[\lambda_i(\overline{E},K')= \inf \{ \lambda \in \bR \ | \ \dim_{K'}(\Vect_{K'} E(\lambda,K')) \geq i\},\]
 and the \textit{Zhang $i$-th minimum} by 
 \[\zeta_i(\overline{E})= \inf \{ \lambda \in \bR \ | \ \dim \mathrm{Zar}( E(\lambda,\overline{\bQ})) \geq i\},\]
 where $\mathrm{Zar}( E(\lambda,\overline{\bQ}))$ denotes the Zariski-closure of $E(\lambda,\overline{\bQ})$ in $E \otimes_K \overline{\bQ}$. Note that $\lambda_i(\overline{E}, \overline{\bQ}) \leq \zeta_i(\overline{E})$ and $\lambda_i(\overline{E}, \overline{\bQ})  \leq \lambda_i(\overline{E}, K')$.

As observed by Gaudron and Rémond \cite[section 3]{GR13}, Zhang's theorem \ref{thmZhangintro} leads to important refinements of absolute Siegel's lemmas previously obtained by Roy and Thunder \cite{RoyThunder}.  Using a similar point of view, Chen \cite{Chencompare, Cheniccm} presented some important consequences of Conjecture \ref{conjChen}, especially towards the \emph{absolute transference problem} (see subsection \ref{paragabstrintro} below). In section \ref{sectionGN} we slightly strengthen Chen's approach by combining the two parts of Theorem \ref{thmmain2intro}. The key result of section \ref{sectionGN} is the following theorem, which relates Zhang's minima to the maximal and minimal slopes of the symmetric powers $S^n (\overline{E}^\vee)$. 
\begin{theorem}\label{thmminslopesintro} We have 
 \[\zeta_d(\overline{E}) = \lim_{n \rightarrow +\infty} \frac{\pmax(S^n (\overline{E}^\vee))}{n}
\ \text{ and } \   \zeta_1(\overline{E}) = \lim_{n \rightarrow +\infty} \frac{\pmin(S^n (\overline{E}^\vee))}{n}.\]
\end{theorem}
 We shall present applications of this new interpretation of Zhang's minima to two classical problems in geometry of numbers. 
 \subsubsection{An absolute transference theorem}\label{paragabstrintro}  An important concern about successive minima is their behaviour with respect to duality. It is well-known and not difficult to see that 
 \[0 \leq \lambda_{i}(\overline{E},K') + \lambda_{d+1-i}(\overline{E}^\vee,K')\]
 for any $i \in \{1, \ldots, d\}$. The so-called transference problem asks for upper bounds for these sums, which are notoriously harder to obtain. When $K' = K = \bQ$, a celebrated theorem of Banaszczyk \cite{Banaszczyk} gives the inequality
 \begin{equation}\label{eqBanintro}
 \lambda_{i}(\overline{E},\bQ) + \lambda_{d+1-i}(\overline{E}^\vee,\bQ) \leq \ln(d),
 \end{equation}
which is highly non-trivial and optimal up to an additive constant. Transference theorems are famous for their applications in lattice-based cryptography, and are also important in Diophantine geometry and transcendental number theory (see for example  \cite{BosserGaudron, GRisogenies}).  It is natural and significant to look for generalizations of \eqref{eqBanintro} to other fields $K, K'$. When $K' = K$, Gaudron \cite[Theorem 36]{Gaudronart18} deduced from Banaszczyk's theorem the upper bound
 \[\lambda_{i}(\overline{E},K) + \lambda_{d+1-i}(\overline{E}^\vee,K) \leq \ln(d) + \frac{1}{[K:\bQ]} \ln |\Delta_{K/\bQ}|,\]
 where  $\Delta_{K/\bQ}$ is the discriminant of $K$. From a Diophantine geometry perspective, the dependence in the discriminant is rather unsatisfactory. In this context, an inequality with $K' =  \overline{\bQ}$ which is independent of the base field $K$ is often much more suitable.   In analogy with Roy--Thunder's approach \cite{RoyThunder}, we call such a bound an absolute transference theorem.  To our knowledge, the best result in this direction is a theorem of Pekker  \cite{Pekker}, giving the inequality
\[ \lambda_{i}(\overline{E},\overline{\bQ}) + \lambda_{d+1-i}(\overline{E}^\vee,\overline{\bQ}) \leq \frac{d-1}{2}.\] 
  It is independent of $K$, but very far from Banaszczyk's one when $K = \bQ$.  The following question was raised by Gaudron \cite[section 4.3]{Gaudronart18}.
\begin{question}[Gaudron]\label{questionGaudron}  Does there exist a polynomial $f \in \bZ[X]$ such that 
 \[\lambda_{i}(\overline{E},\overline{\bQ}) + \lambda_{d+1-i}(\overline{E}^\vee,\overline{\bQ}) \leq \ln(f(d)) \ \ \  \forall \   1 \leq i \leq d = \dim E\]
 for any number field $K$ and any hermitian $K$-vector space $\overline{E}$ ? 
 \end{question} 
  Theorem \ref{thmminslopesintro}  has immediate applications to this problem. The point is that the slopes of $\overline{E}$ behave very well with respect to duality, namely $\pmax(\overline{E}^\vee) = -\pmin(\overline{E})$. 
  Using techniques of Gaudron and Rémond \cite{GR13} to estimate slopes of symmetric powers, we shall derive the following absolute transference theorem from Theorem \ref{thmminslopesintro}. For any $N \in \bN$, we denote by $H_N = 1 + 1/2 + \cdots + 1/N$ the $N$-th harmonic number ($H_0 =0$ by convention).
 
 \begin{theorem}\label{thmtransferenceintro} For any $i \in \{1,\ldots,d \}$, we have
\[ \zeta_{i}(\overline{E}) + \zeta_{d+1-i}(\overline{E}^\vee)\leq H_{i-1} + H_{d-i}.\] 
 \end{theorem}
  Note that 
\[\lambda_{i}(\overline{E},\overline{\bQ}) + \lambda_{d+1-i}(\overline{E}^\vee,\overline{\bQ})\leq \zeta_{i}(\overline{E}) + \zeta_{d+1-i}(\overline{E}^\vee)\]
for any $i \in \{1, \ldots, d\}$.  Since $H_N \leq \ln (2N+1)$ for any $N \geq 0$ \cite{DeTemple}, the upper bound in Theorem \ref{thmtransferenceintro} for $i =1$ is essentially as good as Banaszczyk's.  By the estimate $H_{i-1} + H_{d-i} \leq 2\ln(d)$, Theorem \ref{thmtransferenceintro} answers Question \ref{questionGaudron} and shows that we can choose $f(d) = d^2$.  

\subsubsection{Comparison of successive slopes and minima} For $i \in \{1, \ldots, d\}$, we denote by $\widehat{\mu}_i(\overline{E})$ the $i$-th \textit{successive slope} of $\overline{E}$ (Definition \ref{defimui}). 
 Finding upper bounds for the sums $\zeta_i(\overline{E}) + \widehat{\mu}_i(\overline{E})$  is closely related to the absolute transference problem (see  \cite{Chencompare, Cheniccm} and \cite{Gaudronart18}) and absolute Siegel's lemmas \cite[section 3]{GR13}. For $i = 1$,  it follows from Zhang's theorem \ref{thmZhangintro} that
 \[\zeta_1(\overline{E}) + \widehat{\mu}_1(\overline{E}) \leq (H_d-1)/2\]
 as observed by Gaudron and Rémond \cite[Lemma 3.2]{GR13}. It is significantly harder to deal with the case $i \geq 2$. As for the transference problem, Theorem \ref{thmminslopesintro} opens a new approach to this question, leading to the following statement.
\begin{theorem}\label{corocompminintro} For any $i \in \{1, \ldots, d\}$, we have
\[\zeta_i(\overline{E}) + \widehat{\mu}_i(\overline{E}) \leq H_{d-1}\leq \ln(2d-1).\]
\end{theorem}

 Our methods remain valid when $K$ is a function field, in which case we obtain
\[\zeta_i(\overline{E}) + \zeta_{d-i+1}(\overline{E}^\vee) = \zeta_i(\overline{E}) + \widehat{\mu}_i(\overline{E}) = 0\ \ \forall\  i \in \{1, \ldots, d\}.\]

\subsection{Organization of the paper} We fix some notations and conventions in section \ref{sectionconventions}. We then recall definitions and basic properties of adelic $\bR$-Cartier divisors, including a few facts on arithmetic divisors on arithmetic varieties and  arithmetic  intersection theory that we will need in the proof of Theorem \ref{thmmainintro} (section \ref{sectionadelicRdivisors}). We also define Zhang's successive minima and we give some elementary properties of the essential minimum (see subsection \ref{sectionZmin}). Section \ref{sectionpseudoeff} is devoted to the proof of Theorem \ref{thmmainintro}. After defining asymptotic slopes we prove Theorem \ref{thmmain2intro} in  section \ref{sectionasyslopes}. We recall the definitions of Oukounkov bodies and concave transform in section \ref{sectionBCtransform}. In section \ref{sectionZhang} we prove Corollary \ref{coroEssminmaxcfintro}, Theorem \ref{thmineqZhangintro} and some variants. In section \ref{sectionhermitian} we recall definitions and basic facts about hermitian vector spaces.   Applications of Theorem \ref{thmmain2intro} to geometry of numbers are discussed in section \ref{sectionGN}, with proofs of Theorems \ref{thmminslopesintro}, \ref{thmtransferenceintro} and 
 \ref{corocompminintro}.

\setcounter{tocdepth}{1}
\tableofcontents

\section{Conventions and terminology}\label{sectionconventions}

\subsection{} A scheme is \textit{integral} if it is reduced and irreducible. By \textit{variety} over a field $k$ we mean an integral scheme of finite type over $k$.  If $X$ is a Noetherian integral scheme, we denote by $\Div(X)$ or $\Div(X)_{\bZ}$ the group of Cartier divisors on $X$ and by $\Rat(X)$ the field of rational functions on $X$. If $\mathbb{K}$ denotes $\bQ$ or $\bR$, we let $\Div(X)_{\mathbb{K}} = \Div(X) \otimes_\bZ \mathbb{K}$ and $\Rat(X)_{\mathbb{K}} = \Rat(X) \otimes_\bZ \mathbb{K}$. The elements of $\Div(X)_{\bK}$ and $\Rat(X)_{\bK}$ are called $\bK$-Cartier divisors and $\bK$-rational functions respectively. A $\bK$-rational function  $\phi \in \Rat(X)_{\bK}$ defines a $\bK$-Cartier divisor $(\phi)\in \Div(X)_\bK$. We denote by $\Supp D$ the support of a  $\bK$-Cartier divisor $D$ (see \cite[section 1.2]{MoriwakiMAMS} for details). It is a Zariski-closed subset of $X$ \cite[Proposition 1.2.1]{MoriwakiMAMS}.

\subsection{} Let $K_0$ denote either the field $\bQ$ of rational numbers or the field of functions $k(T)$,  where $k$ is an arbitrary field. A \textit{global field} is by definition a finite extension of $K_0$. It is either a number field or the field of functions of a regular projective curve $C_K$ over $k$ equipped with a finite morphism $\varphi_K \colon C_K \rightarrow C_0 = \bP^1_k$, unique up to $k$-isomorphism. 

\subsection{} Let $K$ be a global field. Let $\Sigma_K$ be the set of places of $K$ and $\Sigma_{K,\infty} \subset K$ the set of archimedean places. Note that $\Sigma_{K,\infty} = \emptyset$ if $K_0 = k(T)$. For each $v \in \Sigma_K$, we denote by $K_v$ the completion of $K$ at the place $v$  and by $K_{0,v}$ the completion of $K_0$ with respect to the restriction of $v$ to $K_0$. We shall normalize absolute values associated to $v$. 

 Assume that $K$ is a number field. For each place $v \in \Sigma_K$, we let $|.|_v$ be the unique absolute value on $K_v$ extending the usual absolute value $|.|_v$ on $\bQ_v$~: $|p|_v = p^{-1}$ if $v$ is a finite place over a prime number $p$, and $|.|_v = |.|$ is the usual absolute value on $\bR$ if $v$ is archimedean. For each $v \in \Sigma_K$,  we let $n_v(K) = [K_v:\bQ_v]/[K:\bQ]$.  
  We now turn to the function field case, so that $K_0 = k(T)$. The set of places of $K$ is in one-to-one correspondence with the set $C_K(\overline{k})$ of closed points of $C_K$ (here $\overline{k}$ denotes an algebraic closure of the base field $k$). For each $v \in \Sigma_K$ and each $f \in K^\times = k(C_K)^\times$, we denote by $\ord_v(f)$ the order of $f$ in the discrete valuation ring $\cO_{C_K,v}$. We consider the absolute value $|.|_v$ on $K$ given by $|f|_v = e^{-\ord_v(f)}$ and we let  \[n_v(K) = \frac{[K_v : K_{0,v}][k(\varphi_K(v)):k]}{[K:K_0]}.\]

  With these conventions, we have the following product formula~:
\begin{equation}\label{productformula}
\forall x \in K^\times, \ \sum_{v \in \Sigma_K}n_v(K)\ln |x|_v = 0.
\end{equation} 
\subsection{} Let $X$ be a scheme on $\Spec K$. For each $v \in \Sigma_K$,  let $X_v$ be the fiber product $X \times_K \Spec K_v$.  If $v$ is non-archimedean, we denote by $X_v^{\mathrm{an}}$ the analytification of $X_v$ in the sense of Berkovich (see \cite[section 1.2]{BPS14} for a short definition sufficient for our purposes). 
 When $K$ is a number field and $v$ is archimedean, we denote by $\Sigma_v \subset K(\bC)$ the set of embeddings $\sigma \colon K \hookrightarrow \bC$ associated to $v$. Hence $\Sigma_v$ is a singleton if $v$ is real and $\Sigma_v$ consists of two conjugate embeddings if $v$ is complex.  For each $\sigma \in \Sigma_v$, we let $K_\sigma = K \otimes_K^\sigma \bC$ be the tensor product of $K$ with respect to $\sigma$ and $X_\sigma = X \times_K^\sigma \Spec \bC$.  
 We define the analytification of $X_v$ by $X_v^{\mathrm{an}} = X_v(\bC) = \sqcup_{\sigma \in \Sigma_v} X_\sigma(\bC)$. The disjoint union 
 \[X_\infty^{\mathrm{an}} :=  \bigsqcup_{v \in \Sigma_K, \ v | \infty} X_v^{\mathrm{an}} = \bigsqcup_{\sigma \in K(\bC)} X_\sigma(\bC) =X(\bC) \]
 is the set of morphisms of $\Spec \bQ$-schemes $x \colon \Spec \bC \rightarrow X$. We let $F_\infty \colon X_\infty^{\mathrm{an}} \rightarrow X_\infty^{\mathrm{an}}$ be the involution induced by the complex conjugation. 
 
 \subsection{}
 Let $x \in X$ be a closed point and $K(x)$ its residue field. For a place $v \in \Sigma_K$, we define the \textit{orbit} $O_v^{\mathrm{an}}(x) \subset X_v^{\mathrm{an}}$ of $x$ as follows. We denote by $\Sigma_v(x)$ the set of $K_v$-algebra morphisms $K(x) \otimes_K K_v \rightarrow \bC_v$, where $\bC_v$ is an algebraic closure of $K_v$.  For each $\sigma \in \Sigma_v(x)$, we denote by $x_\sigma$ the closed point of $X_v$ given by composition
 \[\Spec \bC_v \overset{\sigma}{\longrightarrow} \Spec (K(x) \otimes_K K_v) \overset{x \times \mathrm{id}}{\longrightarrow} X_v.\]
For each $z \in X_v(\bC_v)$, we define a point $z^{\mathrm{an}} \in X_v^{\mathrm{an}}$ as follows. If $v | \infty$, we let $z^{\mathrm{an}} = z$. If $v$ is finite, 
 $z^{\mathrm{an}}\in X_v^{\mathrm{an}}$ is the unique valuation on $K_v(z)$ extending $v$. Finally, we let
 \[O_v^{\mathrm{an}}(x) = \{ x_{\sigma}^{\mathrm{an}} \ | \ \sigma \in \Sigma_v(x)\}.\]
 Note that  the cardinal of the set $O_v^{\mathrm{an}}(x)$ is $[K(x) : K]$.

\subsection{} Assume that $X$ is projective and geometrically integral.  Let $D \in \Div(X)_\bR$, $v \in \Sigma_K$ and let $D_v \in \Div(X_v)_\bR$ be the pullback of $D$ to $X_v$. We consider an open covering $X_v = \cup_{i=1}^{\ell}U_i$  such that $D_v$ is defined by $f_i \in \Rat(X_v)_\bR^\times$ on $U_i$ for each $i \in \{1,\ldots, \ell\}$. A continuous (respectively smooth) $D$\textit{-Green function} on $X_v^{\mathrm{an}}$ is a  function 
\[g_v \colon X_v^{\mathrm{an}}\setminus (\Supp D_v)^{\mathrm{an}} \rightarrow \bR\]
such that $g_v + \ln |f_i|^2_v$ extends to a continuous (respectively smooth) function on the analytification $U_i^{\mathrm{an}}$ of $U_i$ for each $i \in \{1, \ldots, \ell\}$. We refer the reader to \cite[sections 1.4 and 2.1]{MoriwakiMAMS} for more details on Green functions.

\section{Adelic Cartier divisors}\label{sectionadelicRdivisors} 
In this section we define adelic $\bR$-Cartier divisors and recall different notions of positivity. We mainly follow  the book of Moriwaki \cite{MoriwakiMAMS}. 
\subsection{Notations}\label{paragnotationsadelicdiv}

Throughout this section we fix a global field $K$ and a projective, normal and geometrically integral variety $X$ on $\Spec K$. 
 We let $d := \dim X$ and we fix an algebraic closure $\overline{K}$ of $K$.  We define a scheme $\mathcal{S}$ as follows~:
\begin{enumerate}
\item if $K_0 = \bQ$, $\mathcal{S} = \Spec \cO_K$ is the spectrum of the ring of integers $\cO_K$ of $K$;
\item if $K_0 = k(T)$, $\mathcal{S} = C_K$ is a regular projective integral curve over $k$ with field of functions $k(C_K)=K$. This implies the existence of a finite $k$-morphism $\varphi \colon C_K \rightarrow \bP^1_k$. Moreover, the curve $C_K$ is unique up to $k$-isomorphism.
\end{enumerate}
 In all this section, $\mathbb{K}$ denotes either $\bZ$, $\bQ$ or $\bR$.  Let $D \in \Div(X)_\bK$ and let $U \subset \mathcal{S}$ be a non-empty open subset. A \textit{model} $\mathcal{X}$ of $X$ over $U$ is an integral scheme together with a projective dominant morphism $\pi_{\mathcal{X}} \colon \mathcal{X} \rightarrow U$ with generic fiber $X$. We will denote a model $(\mathcal{X}, \pi_{\mathcal{X}})$ if we need to specify the associated morphism. We say that $\mathcal{X}$ is a \textit{normal model} if $\mathcal{X}$ is normal. If $\mathcal{D}$ is a $\bK$-Cartier divisor on $\mathcal{X}$ such that the restriction of $\mathcal{D}$ to $X$ is equal to $D$, we say that $(\mathcal{X},\mathcal{D})$ is a model of $(X,D)$ over $U$. For each non-archimedean place $v \in U$, we denote by $g_{\mathcal{D},v}$ the $D$-Green function on $X_v^{\mathrm{an}}$ induced by $\mathcal{D}$ (see \cite[section 0.2]{MoriwakiMAMS} for details on this construction). 

\subsection{Definitions} 
\begin{defi}\label{defiadelicRdivisor} A \textit{metrized} $\bK$-\textit{Cartier divisor} on $X$ is a pair $\overline{D} = (D, (g_v)_{v \in \Sigma_K})$ consisting of a $\bK$-Cartier divisor $D$ on $X$ and of a continuous $D$-Green function $g_v$ on $X_v^{\mathrm{an}}$ for each $v \in \Sigma_K$. We say that $\overline{D}$ is an \textit{adelic} $\bK$-\textit{Cartier divisor} on $X$ if moreover the following conditions are satisfied.
 \begin{enumerate}
 \item There exists a dense open subset $U$ of $\mathcal{S}$ and a normal model $(\mathcal{X}_U,\mathcal{D}_U)$ of $(X,D)$ over $U$  such that $g_{v}=g_{\mathcal{D}_U,v}$ for all $v \in U$. 
 \item If $K$ is a number field, $g_v$ is invariant under the complex conjugation $F_\infty$ for each $v \in \Sigma_{K,\infty}$.
 \end{enumerate} 
 The set of adelic $\bK$-Cartier divisors on $X$ forms a group, denoted by $\widehat{\mathrm{Div}}(X)_\bK$. Note that
 $\widehat{\mathrm{Div}}(X)_\bZ \subset \widehat{\mathrm{Div}}(X)_\bQ \subset \widehat{\mathrm{Div}}(X)_\bR$.
 In the sequel, the elements of $\widehat{\Div}(X)_\bZ$ will sometimes be called  \textit{adelic Cartier divisors} for simplicity. 
\end{defi}

\begin{rema} In \cite[Definition 4.1.1]{MoriwakiMAMS}, adelic $\bK$-Cartier divisors are called adelic arithmetic $\bK$-Cartier divisor of $C^0$-type . \end{rema}

Note that an adelic $\bR$-Cartier divisor $\overline{\xi}$ on $\Spec K$ is just a collection of real numbers $\overline{\xi} = (\xi_v)_{v \in \Sigma_K}$ such that $\xi_v =0$ for all but finitely many $v \in \Sigma_K$. The \textit{normalized Arakelov degree} of $\overline{\xi}$ is defined by
\[\widehat{\deg}(\overline{\xi}) = \frac{1}{2}\sum_{v \in \Sigma_K}n_v(K) \xi_v.\]

 \subsection{Arithmetic volume function}\label{sectionvolume}

Let $\overline{D} = (D, (g_v)_{v \in \Sigma_K})$ be an adelic $\bR$-divisor on $X$. We consider the $K$-vector space 
\[H^0(X,D) := \{ \phi \in \Rat(X)^\times \ | \ D + \phi \geq 0\} \cup \{0\}.\] 
 For any non-zero element $\phi \in H^0(X,D)$ and any $v \in \Sigma_K$, the function $\|\phi\|_{v}(x) := |\phi(x)|_v\exp(-g_v(x)/2)$ extends to a continuous function on $X_v^{\mathrm{an}}$ (see \cite[Propositions 1.4.2 and 2.1.3]{MoriwakiMAMS}). We also define $\|\phi\|_{v,\sup} := \sup_{x \in X^{\mathrm{an}}} \|\phi\|_{v}(x)$ and
\[\widehat{H}^0(X,\overline{D}) := \{ \phi \in H^0(X,D)\ | \ \|\phi\|_{v,\sup}\leq 1 \ \forall v \in \Sigma_K \}.\] 
 We let
\[\widehat{h}^0(X,\overline{D}) = \left\{\begin{tabular}{ll}
$\ln \# \widehat{H}^0(X,\overline{D})$ & if $K_0 = \bQ$,\\
$\dim_k \widehat{H}^0(X,\overline{D})$ & if $K_0 = k(T)$.
\end{tabular} \right.\]
\begin{defi}The \textit{arithmetic volume} of $\overline{D}$ is the quantity
\[\widehat{\vol}(\overline{D}) := \limsup_{n \rightarrow +\infty} \frac{\widehat{h}^0(X,n\overline{D})}{n^{d+1}/(d+1)!}.\]
\end{defi} 

\begin{example}\label{examplevolcdf} Assume that there exists a normal model $(\mathcal{X},\mathcal{D})$ over $\mathcal{S}$ such that all the $D$-Green functions of $\overline{D}$ are induced by $\mathcal{D}$. By \cite[Proposition 4.3.1]{MoriwakiMAMS}, we have
\[\widehat{H}^0(X,D) = \{ \phi\in H^0(\mathcal{X},\mathcal{D}) \ | \ \max_{v\in \Sigma_{K,\infty}}\|\phi\|_{v,\sup}\leq 1\}.\]
In particular, if $K_0 = k(T)$ then $\widehat{H}^0(X,D) = H^0(\mathcal{X},\mathcal{D})$. It follows that $\widehat{\vol}(\overline{D})=\vol(\mathcal{D})$ is the (geometric) volume of $\mathcal{D}$ (see \cite[section 2.2]{Laz})~:  
\[\vol(\mathcal{D})= \limsup_{n \rightarrow +\infty} \frac{\dim_k H^0(\mathcal{X}, n\mathcal{D})}{n^{d+1}/(d+1)!}.\] 
\end{example}

 We now recall a well-known continuity property of the volume function $\widehat{\vol}$ due to Moriwaki.
\begin{theorem}\label{thmcontvol} Let $\overline{D},\overline{D}_1, \ldots, \overline{D}_\ell$ be adelic Cartier divisors on $X$. Then 
\[\lim_{\varepsilon_1 \rightarrow 0, \ldots, \varepsilon_\ell \rightarrow 0} \widehat{\vol}(\overline{D}+\varepsilon_1\overline{D}_1 + \cdots + \varepsilon_\ell\overline{D}_\ell) = \widehat{\vol}(\overline{D}).\]
\end{theorem}
 When $K$ is a number field, this is a particular case of \cite[Theorem 5.2.1]{MoriwakiMAMS}. In the function field case, one can prove the result by using similar arguments and the continuity of the geometric volume function \cite[Theorem 2.2.44]{Laz}. Alternatively, Theorem \ref{thmcontvol} is a particular case of \cite[Theorem 6.4.24]{CMadeliccurves} due to Chen and Moriwaki, who recently established the continuity of the arithmetic volume in the much more general setting of adelic curves.

\subsection{Height function}\label{sectionheightdiv}
 Let $\overline{D} = (D, (g_v)_{v \in \Sigma_K})$ be an adelic $\bR$-Cartier divisor on $X$ and let $x \in X(\overline{K})$ be a closed point with residue field $K(x)$. The \textit{height} of a $x$ with respect to $\overline{D}$ is defined to be
\[h_{\overline{D}}(x):= -\frac{1}{[K(x) : K]}\sum_{v \in \Sigma_K}n_v(K) \sum_{z \in O_v^{\mathrm{an}}(x)} \ln \|\phi\|_v(z),\]
where $\phi \in H^0(X,D)$ is any function with $x \notin \Supp(D+(\phi))$. This definition does not depend on the choice of $\phi$ (see \cite[section 4.2]{MoriwakiMAMS} for details). 
\begin{rema} We mention an equivalent formula for the height of $x$, which may be more usual. For each place $w \in \Sigma_{K(x)}$, we fix a $K$-embedding $\sigma_w \colon K(x) \hookrightarrow \bC_v$ associated to $w$, where $v$ denotes the restriction of $w$ to $K$ (note that there are exactly $[K(x)_w : K_v]$ such embeddings). The pair $(x,\sigma_w)$ determines uniquely a point $x_w \in X_v$, and the quantity $\|\phi\|_w(x) := \|\phi\|_v(x_w^{\mathrm{an}})$ does not depend on the choice of $\sigma_w$. With these notations, we have
\[h_{\overline{D}}(x) = -\sum_{w \in \Sigma_{K(x)}} n_w(K(x)) \ln \|\phi\|_w(x).\]
\end{rema}

If we denote by  $\pi \colon X \rightarrow \Spec K$ the structural morphism, we have
\begin{equation}\label{projformulaheight}
h_{\overline{D} + \pi^*\overline{\xi}}(x) = h_{\overline{D}}(x) + \widehat{\deg}(\overline{\xi})
\end{equation}
for any $\overline{\xi} \in \widehat{\Div}(\Spec K)_\bR$ and $x \in X(\overline{K})$. In view of example \ref{exampleheightintersff} below,  this basic fact can be interpreted as a projection formula. 
\begin{example}\label{exampleheightintersff} Assume that $K$ is a function field and that there exists a normal model $(\mathcal{X}, \mathcal{D})$ of $(X,D)$ over $\mathcal{S} = C_K$ such that for each $v \in \Sigma_K$, the $D$-Green function $g_v = g_{\mathcal{D},v}$ is induced by $\mathcal{D}$. Let $x \in X(\overline{K})$ be a closed point and let $C_x \subset \mathcal{X}$ be the Zariski-closure of $x$ in $\mathcal{X}$. Note that $C_x$ is an integral projective curve, and that the restriction ${\pi_{\mathcal{X}}}_{|C_x} \colon C_x \rightarrow C_K$ is surjective. We have 
\[h_{\overline{D}}(x) = \frac{\mathcal{D} \cdot C_x}{[K(x) : K_0]},\]
where $K(x)$ is the residue field of $X$ at $x$ and $ \mathcal{D} \cdot C_x$ is the usual geometric intersection product of cycles. We will see that an analogue of this formula involving the arithmetic intersection product holds when $K$ is a number field (see section \ref{sectionintersection}).
\end{example}

We end this paragraph with the following well-known lemma, giving a lower bound for the height of points at which a small section doesn't vanish.

\begin{lemma}\label{lemmaheightsupport} Suppose there exists $\phi \in {H}^0(X,D) \setminus \{0\}$ with $\|\phi\|_{v,\sup} \leq 1$ for all $v \in \Sigma_K$. For any $x \in X(\overline{K})$ not contained in the support of $D + (\phi)$, we have $h_{\overline{D}}(x) \geq 0$.
\end{lemma}
 
 \begin{proof}
  If $x \in X(\overline{K})$ is not contained in the support of $D + \divi(\phi)$, we have 
  \[h_{\overline{D}}(x) = - \frac{1}{[K(x):K]}\sum_{v \in \Sigma_K} \sum_{z \in O_v^{\mathrm{an}}(x)} \ln \|\phi\|_v(z) \geq 0\]
since $\|\phi\|_v(z) \leq \|\phi\|_{v,\sup} \leq 1$ for every $v \in \Sigma_K$ and $z \in O_v^{\mathrm{an}}(x)$.
 \end{proof}

\subsection{Positivity of adelic $\bR$-Cartier divisors}

In this section we define several notions of arithmetic positivity, following \cite[section 4.4]{MoriwakiMAMS} and \cite{Zhangadelic}. For the definition of plurisubharmonic functions we refer the reader to \cite[section 2.1]{MoriwakiZariski} and \cite[section 1.4]{MoriwakiMAMS}.
\begin{defi}\label{defipositivity}
 Let $\overline{D} = (D, (g_v)_{v \in \Sigma_K})$ be an adelic $\bR$-Cartier divisor on $X$.  We say that $\overline{D}$ is
\begin{itemize}
\item 
 \textit{big} if $\widehat{\vol}(\overline{D}) > 0$;
\item 
 \textit{pseudo-effective} if $\overline{D} + \overline{A}$ is big for any big adelic $\bR$-Cartier divisor $\overline{A}$ on $X$;
\item \textit{semi-positive} if there exists a sequence $(\mathcal{X}_n, \mathcal{D}_n, (g_{n,v})_{v \in \Sigma_K})_{n \in \bN}$ such that~:
\begin{enumerate}
\item for all $n \in \bN$,  $(\mathcal{X}_n, \mathcal{D}_n)$ is a normal model for $(X,D)$ with $\mathcal{D}_n$ relatively nef, 
\item for all $n \in \bN$, $g_{n,v}$ is a smooth plurisubharmonic $D$-Green function invariant under $F_\infty$ if $v \in \Sigma_{K,\infty}$ and $g_{n,v}=g_{\mathcal{D}_n,v}$ for every non-archimedean $v \in \Sigma_K$,
\item for every $v \in \Sigma_K$, $(g_{n,v})_{n \in \bN}$ converges uniformly to $g_v$;
\end{enumerate}
 \item \textit{nef} if $\overline{D}$ is semi-positive and if $\inf_{x \in X(\overline{K})} h_{\overline{D}}(x) \geq 0$;
\end{itemize}
\end{defi}

\begin{rema} (1) In the definition of semi-positivity, the condition that $\mathcal{D}_n$ is relatively nef is to be understood with respect to the morphism $\pi_{\mathcal{X}_n} \colon \mathcal{X}_n \rightarrow \mathcal{S}$ associated to the model $\mathcal{X}_n$~: by definition, $\mathcal{D}_n$ is relatively nef if it has non-negative intersection with any curve contained in a fiber of $\pi_{\mathcal{X}_n}$ above a closed point.

(2) When $\overline{D} \in \widehat{\Div}(X)_{\bZ}$ and $K$ is a number field, our definition of semi-positivity coincides with the notion of semi-positive adelic line bundles in the sense of Zhang \cite{Zhangadelic} (see \cite{BMPS}, (1) page 229). In \cite{MoriwakiMAMS}, Moriwaki introduced the more intrinsic notion of \textit{relatively nef} adelic $\bR$-Cartier divisors, which is  closely related to semi-positivity in the sense of Definition \ref{defipositivity}. By \cite[Proposition 4.2.2]{MoriwakiMAMS}, if $\overline{D}\in \widehat{\Div}(X)_{\bR}$ is relatively nef in the sense of \cite[Definition 4.4.1]{MoriwakiMAMS}, then it is semi-positive (see also \cite{BMPS}, (2) and (4) page 229). 
\end{rema}

\begin{example}\label{examplepseffcdf} In the situation of example \ref{examplevolcdf} when $K_0=k(T)$, $\overline{D}$ is big (respectively pseudo-effective) if and only if $\mathcal{D}$ is big (respectively pseudo-effective) in the sense of \cite[section 2.2]{Laz}.
\end{example}

\subsection{Arithmetic Cartier divisors} 

 In this section we define arithmetic $\bK$-Cartier divisors on arithmetic varieties, which are special cases of adelic $\bK$-Cartier divisors for which all the Green functions at the non-archimedean places are induced by the same model.  In this paragraph we assume that $K$ is a number field.  Let $\mathcal{X}$ be a projective arithmetic variety over $\Spec \cO_K$, that is an integral scheme projective and flat over $\Spec \cO_K$. We let $d+1$ be the Krull dimension of $\mathcal{X}$, so that the generic fiber $X := \mathcal{X}\times_{\cO_K} \Spec K$ of $\mathcal{X}$ has dimension $d$.
 We say that $\mathcal{X}$ is generically smooth if $X$ is smooth.
\begin{defi}
 An \textit{arithmetic} $\bK$\textit{-Cartier divisor} $\overline{\mathcal{D}} = (\mathcal{D},(g_v)_{v \in \Sigma_{K,\infty}})$ on $\mathcal{X}$ is a pair  consisting of a $\bK$-Cartier divisor $\mathcal{D} \in \Div(\mathcal{X})_\bK$ and, for each archimedean place $v \in \Sigma_{K,\infty}$, a continuous $\mathcal{D}$-Green function $g_v$ on $X_v^{\mathrm{an}}$ invariant under the complex conjugation $F_\infty$. We say that $\overline{\mathcal{D}}$ if of $C^\infty$-type if $g_v$ is smooth for every $v \in \Sigma_{K,\infty}$. 
\end{defi}  
Arithmetic $\bK$-Cartier divisors on $\mathcal{X}$ form a group, denoted by $\widehat{\Div}(\mathcal{X})_\bK$. Given an arithmetic $\bK$-Cartier divisor $\overline{\mathcal{D}} = (\mathcal{D},(g_v)_{v \in \Sigma_{K,\infty}})$, we denote by $\overline{\mathcal{D}}^{\mathrm{ad}}$ the adelic $\bK$-Cartier divisor on the generic fiber $X$ defined by $(D,(g_v)_{v \in \Sigma_K})$, where $D = \mathcal{D}_{|X}$ is the restriction of $\mathcal{D}$ and for every non-archimedean place $v \in \Sigma_K$, $g_v = g_{\mathcal{D},v}$ is the $D$-Green function induced by $\mathcal{D}$. 
\begin{defi} An arithmetic $\bR$-Cartier divisor $\overline{\mathcal{D}}$ on $\mathcal{X}$ is said to be \textit{big} (respectively \textit{pseudo-effective}) if $\overline{\mathcal{D}}^{\mathrm{ad}}$ is big (respectively pseudo-effective). We say that  $\overline{\mathcal{D}}$ is \textit{semi-positive} if $\mathcal{D}$ is relatively nef and $g_v$ is plurisubharmonic for every $v \in \Sigma_{K,\infty}$, and that $\overline{\mathcal{D}}$ is \textit{nef} if it is semi-positive with $h_{\overline{\mathcal{D}}^{\mathrm{ad}}}(x) \geq 0$ for every $x \in X(\overline{K})$.
\end{defi}
We observe that our definitions of big, pseudo-effective and nef arithmetic $\bR$-divisor coincide with the ones of Ikoma \cite{Ikoma15}. We will also need the notion of ampleness used in \cite{Ikoma15}. We first recall the definition of  the curvature current. Let $\overline{\mathcal{D}} = (\mathcal{D},(g_v)_{v}) \in \widehat{\Div}(\mathcal{X})_\bZ$.  We assume that $\mathcal{X}$ is generically smooth.  For $v \in \Sigma_{K, \infty}$, we denote by $D_v$ the divisor induced by $\mathcal{D}$ on $X_v$ and we consider the current $c_1(\overline{\mathcal{D}}_v) := \frac{\sqrt{-1}}{2\pi}\partial\overline{\partial}[g_v] + \delta_{D_v}$ on $X_v^{\mathrm{an}}$. Note that $c_1(\overline{\mathcal{D}}_v)$ is non-negative if the function $g_v$ is plurisubharmonic, and that  it  is a real $(1,1)$ form if $g_v$ is smooth. We let $c_1(\overline{\mathcal{D}})$ be the current on $X(\bC)$ whose restriction to $X_v^{\mathrm{an}}$ is $c_1(\overline{\mathcal{D}}_v)$ for each $v \in \Sigma_{K,\infty}$. 

The following definition of ampleness is the one used by Ikoma in \cite{Ikoma15}.
\begin{defi} Let $\mathcal{X}$ be a generically smooth arithmetic variety over $\Spec \cO_K$ and let  $\overline{\mathcal{A}} = (\mathcal{A}, (g_v)_{v \in \Sigma_{K,\infty}}) \in \widehat{\Div}(\mathcal{X})_{\bZ}$. We say that $\overline{\mathcal{A}}$ is ample if it satisfies the following conditions~:
\begin{enumerate}
\item $\mathcal{A}$ is ample;
\item for each $v \in \Sigma_{K,\infty}$, $g_v$ is smooth;
\item the curvature form $c_1(\mathcal{\overline{A}})$ is positive pointwise on $X(\bC)$;
\item for $m\gg1$, ${H}^0(\mathcal{X},m\overline{\mathcal{A}})$ is generated as an  $\cO_K$-module by sections $s$ with $\max_{v\in \Sigma_{K,\infty}}\|s\|_{v,\sup} < 1$.
\end{enumerate}
An arithmetic $\bR$-divisor $\overline{\mathcal{D}}$ on $\mathcal{X}$ is \textit{ample} if there exist ample arithmetic $\bZ$-Cartier divisors $\overline{\mathcal{A}}_1, \ldots, \overline{\mathcal{A}}_\ell$ and positive real numbers $a_1, \ldots, a_\ell$ such that
\[\overline{\mathcal{D}} = a_1\overline{\mathcal{A}}_1+ \cdots + a_\ell\overline{\mathcal{A}}_\ell.\]
\end{defi} 

\subsection{Reminder on arithmetic intersection theory}\label{sectionintersection} In this paragraph we recall some properties of the arithmetic intersection product that we will need in the proof of Theorem \ref{thmmainintro}. We assume that $K$ is a number field and let $\mathcal{X}$  be a normal  arithmetic variety over $\Spec \cO_K$.

We say that a divisor $\overline{\mathcal{D}} \in \widehat{\Div}(\mathcal{X})_\bR$ is \textit{integrable} if it belongs to the subgroup of $\widehat{\Div}(\mathcal{X})_\bR$ generated by nef arithmetic $\bR$-Cartier divisors. In particular, a semi-positive arithmetic $\bR$-Cartier divisor is integrable. Let $\overline{\mathcal{D}}_0,\overline{\mathcal{D}}_1, \ldots, \overline{\mathcal{D}}_d$ be integrable arithmetic $\bR$-Cartier divisors on $\mathcal{X}$. In \cite[section 6.4]{MoriwakiZariski}, Moriwaki defined a multilinear intersection product 
\begin{equation}\label{eqintersectproduct}
\widehat{\deg}(\overline{\mathcal{D}}_0 \cdots \overline{\mathcal{D}}_d)
\end{equation}
when $\mathcal{X}$ generically smooth. In \cite{Ikoma15}, Ikoma extended the intersection product \eqref{eqintersectproduct} to the cases where $\mathcal{X}$ is not generically smooth and $\overline{\mathcal{D}}_0$ is not integrable. Let $\overline{\mathcal{D}} \in \widehat{\Div}(\mathcal{X})_\bR$ be an integrable arithmetic $\bR$-Cartier divisor and $\overline{D} = \mathcal{\overline{D}}^{\mathrm{ad}}$ be the induced adelic divisor on $X$. Let $x \in X(\overline{K})$ be a closed point and let $C_x \subset \mathcal{X}$ be the Zariski closure of $x$ in $\mathcal{X}$. Then we have
\begin{equation}\label{eqheightinters}
h_{\overline{D}}(x) = \frac{\widehat{\deg}(\overline{\mathcal{D}}_{|C_x})}{[K(x): K_0]}
\end{equation}
(see \cite[section 5.3]{MoriwakiZariski} for details). When $\mathcal{X}$ is generically smooth and when the divisors $\overline{\mathcal{D}}_0,\overline{\mathcal{D}}_1, \ldots, \overline{\mathcal{D}}_d \in \widehat{\Div}(\mathcal{X})_\bZ$ are of $C^{\infty}$-type, then the product \eqref{eqintersectproduct} coincides with the usual one,  used for example in \cite{Zhangplav, Zhangadelic}.  In particular, for any non-zero function  $\phi \in H^0(\mathcal{X}, \mathcal{D}_d)$ we have  
\begin{multline*}
\widehat{\deg}(\overline{\mathcal{D}}_0 \cdots \overline{\mathcal{D}}_d) = \sum_{i=1}^\ell a_i \widehat{\deg}({\overline{\mathcal{D}}_0}_{|\mathcal{Z}_i} \cdots {\overline{\mathcal{D}}_{d-1}}_{|\mathcal{Z}_i})\\
- \sum_{v \in \Sigma_{K,\infty}} \int_{X_v^{\mathrm{an}}} \ln\|\phi\|_{v} c_{1,v}(\overline{\mathcal{D}}_0) \wedge \cdots c_{1,v}(\overline{\mathcal{D}}_{d-1}),
\end{multline*}
where $(\phi) +{\mathcal{D}}_d  = \sum_{i=1}^\ell a_i \mathcal{Z}_i$ is the decomposition of the Weil divisor $(\phi)+{\mathcal{D}}_d $~: for each $i \in \{1, \ldots, \ell\}$, $a_i \in \bZ_{> 0}$ and $\mathcal{Z}_i \subset \mathcal{X}$ is a codimension $1$ integral subvariety. 

\subsection{Successive minima of an adelic $\bR$-Cartier divisor}\label{sectionZmin}

Let $K$ be global field and let $X$ be a projective, normal and geometrically integral variety on $\Spec K$ of dimension $d \geq 1$. Let $\overline{D} = (D, (g_v)_{v \in \Sigma_K})$ be an adelic $\bR$-Cartier divisor on $X$. For all $\lambda \in \bR$, we denote by $X_{\overline{D}}(\lambda)$ the Zariski-closure in $X$ of the set 
\[\{x \in X(\overline{K}) \ | \ h_{\overline{D}}(x) \leq \lambda\}.\]
For every integer $i \in \{1, \ldots, d +1\}$, we define the $i$\textit{-th minimum} $\zeta_i(\overline{D})$ of $\overline{D}$ by 
\[\zeta_i(\overline{D}) = \inf \{\lambda \in \bR \ | \ \dim X_{\overline{D}}(\lambda) \geq i-1 \}.\]
These invariants were first introduced by Zhang in the number field setting \cite{Zhangplav}.  We call $\Essabs(\overline{D}) = \zeta_1(\overline{D})$ the \textit{absolute minimum} and $\Essmin(\overline{D}) := \zeta_{d+1}(\overline{D})$ the \textit{essential minimum}  of $\overline{D}$. Note that $\Essabs(\overline{D}) = \inf_{x \in X(\overline{K})} h_{\overline{D}}(x)$ and
\[\Essmin(\overline{D}) = \inf \{\lambda \in \bR \ | \ X_{\overline{D}}(\lambda) \text{ is dense} \} = \sup_{Y \subsetneq X}\inf_{x \in X(\overline{K}) \setminus Y} h_{\overline{D}}(x),\]
where the supremum is taken on the Zariski-closed proper subschemes $Y$ of $X$. 

We gather some basic properties of the essential minimum in the following lemma.
\begin{lemma}\label{propertiesEssmin} Let $\overline{D}_1$, $\overline{D}_2$ be two adelic $\bR$-Cartier divisors on $X$.
\begin{enumerate}
\item\label{supaddEssmin} (super-additivity) We have
\[\Essmin(\overline{D}_1+ \overline{D}_2) \geq  \Essmin(\overline{D}_1)+\Essmin( \overline{D}_2).\]
\item\label{nondecreasingEssmin} (non-decreasing) If $\widehat{H}^0(X,t(\overline{D}_1-\overline{D}_2)) \ne \{0\}$ for some real number $t > 0$, then  $\Essmin(\overline{D}_1)\geq\Essmin( \overline{D}_2)$.
\item\label{contEssmin} (continuity) If the underlying Cartier divisor $D_1$ of $\overline{D}_1$ is big, we have
\[ \lim_{t \rightarrow 0} \Essmin(\overline{D}_1 + t \overline{D}_2) = \Essmin(\overline{D_1}). \]
\end{enumerate}
\end{lemma}
\begin{proof}
\eqref{supaddEssmin} Without loss of generality, we assume that $ \Essmin(\overline{D}_1),\Essmin( \overline{D}_2) > -\infty$.  
  Let $\lambda_1 < \Essmin(\overline{D}_1)$ and $\lambda_2 < \Essmin(\overline{D}_2)$ be real numbers. By definition of the essential minimum, there exist non-empty open subsets $U_1$, $U_2$ in $X$ such that for each $i \in \{1,2\}$, 
\[\forall x \in U_i, \ h_{\overline{D}_i}(x) \geq \lambda_i.\]
Since $X$ is integral, the open subset $U = U_1 \cap U_2$ is non-empty. It follows that 
\[\Essmin(\overline{D}_1+ \overline{D}_2) \geq \inf_{x \in U(\overline{K})} (h_{\overline{D}_1}(x) +h_{\overline{D}_2}(x)) \geq \lambda_1 + \lambda_2,\]
and we conclude by letting $\lambda_1$, $\lambda_2$ tend to $\Essmin(\overline{D}_1)$, $\Essmin(\overline{D}_2)$
respectively.

\eqref{nondecreasingEssmin} By \eqref{supaddEssmin}, we have
\[\Essmin(\overline{D}_1) =\Essmin(\overline{D}_2 + \overline{D}_1 - \overline{D}_2) \geq \Essmin(\overline{D}_2) + \Essmin( \overline{D}_1-\overline{D}_2).\]
By Lemma \ref{lemmaheightsupport}, there exists a closed subset $Z_t \subsetneq X$ such that  $th_{\overline{D}_1-\overline{D}_2}(x) = h_{t(\overline{D}_1-\overline{D}_2)}(x) \geq 0 $ for any $x \in X(\overline{K})\setminus Z_t$. Therefore
\[\Essmin( \overline{D}_1-\overline{D}_2) \geq \inf_{x \in X(\overline{K}) \setminus Z_t} h_{\overline{D}_1-\overline{D}_2}(x) \geq 0.\]

\eqref{contEssmin} Let $\pi \colon X \rightarrow \Spec K$ denote the structural morphism. If we replace $\overline{D}_1$ by $\overline{D}_1 + \pi^*\overline{\xi}$ for some adelic $\bR$-Cartier divisor $\overline{\xi}$ on $\Spec K$, both sides of the inequality differ by $\widehat{\deg}(\overline{\xi})$ (see the projection formula for the height \eqref{projformulaheight}). Hence we may assume that $\overline{D}_1$ is big. Let $\varepsilon > 0$ and $t$ be real numbers. By continuity of the volume function (Theorem \ref{thmcontvol}), the divisors
\[(1+\varepsilon)\overline{D}_1 - (\overline{D}_1 + t\overline{D}_2) = \varepsilon\overline{D}_1 - t\overline{D}_2\]
and 
\[(\overline{D}_1 + t\overline{D}_2)-(1-\varepsilon)\overline{D}_1 = \varepsilon\overline{D}_1 + t\overline{D}_2\]
are both big if $t$ is sufficiently close to zero. By \eqref{nondecreasingEssmin}, we have 
\[(1+\varepsilon)\Essmin(\overline{D}_1) \geq \Essmin(\overline{D}_1 + t\overline{D}_2) \geq (1-\varepsilon)\Essmin(\overline{D}_1)\]
and the result follows.
\end{proof}

\begin{lemma}\label{lemmaineqessmin}
 If $\overline{D}$ is pseudo-effective and $D$ is big, then $\Essmin (\overline{D}) \geq 0$.
\end{lemma}
\begin{proof}
 If $\overline{D}$ is big, then $\Essmin(\overline{D}) \geq 0$ by Lemma \ref{propertiesEssmin} \eqref{nondecreasingEssmin}. We deduce the general case by continuity (Lemma \ref{propertiesEssmin} \eqref{contEssmin}).
\end{proof}

\section{Pseudo-effectivity and essential minimum}\label{sectionpseudoeff}

 Let $K$ be a global field and let $\mathcal{S} = C_K$ or $\Spec \cO_K$ denote the corresponding scheme defined in section \ref{paragnotationsadelicdiv}. Let $X$ be a projective, normal and geometrically integral variety $X$ on $\Spec K$ and let $\overline{D} =(D,(g_v)_{v \in \Sigma_K})$ be an adelic $\bR$-Cartier divisor on $X$. The purpose of this section is to prove Theorem \ref{thmmainintro} in the introduction. We reproduce the statement below.

\begin{theorem}\label{thmmain}
Suppose that $\overline{D}$ is semi-positive and that $D$ is big. Then $\overline{D}$ is pseudo-effective if and only if $\Essmin(\overline{D}) \geq 0$.
\end{theorem}

 By Lemma \ref{lemmaineqessmin}, we only need to prove the implication \[\Essmin(\overline{D}) \geq 0 \implies \overline{D} \text{ is pseudo-effective.}\]
 We first reduce the problem to the case $\overline{D} = \overline{\mathcal{D}}^{\mathrm{ad}}$ (subsection \ref{paragfixedmodel})~: we shall see that we may assume the existence of a normal model $(\mathcal{X},\mathcal{D})$ on $\mathcal{S}$ such that for every non-archimedean place $v \in \Sigma_K$, the  $D$-Green function $g_v$ is induced by $\mathcal{D}$. In this case, we can interpret the height of a closed point as an intersection number on $\mathcal{X}$ (see example \ref{exampleheightintersff} and subsection \ref{sectionintersection}). When $\mathcal{S} = C_K$, we will deduce Theorem \ref{thmmain} from a deep theorem of Boucksom--Demailly--P\u{a}un--Peternell \cite{BDPP}, characterizing pseudo-effectivity  in terms of numerical positivity (see Theorem \ref{thmBDPP} below). When $\mathcal{S} = \Spec \cO_K$, we adopt the same strategy and use the arithmetic counterpart of Theorem \ref{thmBDPP}, due to Ikoma \cite{Ikoma15}.  To clarify the exposition, we shall distinguish the cases $\mathcal{S} = C_K$ (subsection \ref{paragffcase}) and $\mathcal{S} = \Spec \cO_K$ (subsection \ref{paragnfcase}). 
 
 Before we start the proof, we mention a useful corollary of Theorem \ref{thmmain}.
 \begin{coro}\label{coromain} 
 Let $\overline{\xi}$ be an adelic $\bR$-Cartier divisor on $\Spec K$ with $\widehat{\deg}(\overline{\xi}) = 1$ and let $\pi \colon X \rightarrow \Spec K$ be the structural morphism.  Suppose that $D$ is big. Then we have
\[\Essmin(\overline{D}) \geq \sup\{t \in \bR \ | \ \overline{D} - t\pi^*\overline{\xi}  \text{ is pseudo-effective}\},\]
with equality if $\overline{D}$ is semi-positive.
 \end{coro}
 
\begin{proof}
 For $t \in \bR$, we let $\overline{D}_t = \overline{D} - t \pi^*\overline{\xi}$.   By the projection formula for the height \eqref{projformulaheight}, we have $\Essmin(\overline{D}_t) = \Essmin(\overline{D}) - t$. By Lemma \ref{lemmaineqessmin}, $\Essmin(\overline{D}) \geq t$ if $\overline{D}_t$ is pseudo-effective. It follows that
 \begin{equation}\label{ineqEssmin0}
 \Essmin(\overline{D}) \geq \sup\{t \in \bR \ | \ \overline{D}_t  \text{ is pseudo-effective}\}.
 \end{equation}
If moreover $\overline{D}$ is semi-positive, then $\overline{D}_t$ is semi-positive for any $t \in \bR$. By  Theorem \ref{thmmain}, $\overline{D}_t$ is pseudo-effective for any $t \leq \Essmin(\overline{D})$. Hence \eqref{ineqEssmin0} is an equality.
\end{proof}

\subsection{Reduction to a fixed model}\label{paragfixedmodel} 

\begin{lemma}\label{lemmaapproxdiv}
 Assume that $\overline{D} = (D,(g_{v})_{v \in \Sigma_K})$ is semi-positive but not pseudo-effective. Then there exist a normal model $(\mathcal{X}, \mathcal{D})$ of $(X,D)$ on $\mathcal{S}$ together with a collection $(g_{\mathcal{D},v})_{v \in \Sigma_{K,\infty}}$  such that
    \begin{itemize}
    \item $\mathcal{D} $ is relatively nef, 
    \item for all $v \in \Sigma_{K,\infty}$, $g_{\mathcal{D},v}$ is a smooth $D$-Green function invariant by $F_\infty$,
    \item  $\overline{\mathcal{D}} $ is not pseudo-effective,
    \item $ \Essmin(\overline{\mathcal{D}}^{\mathrm{ad}}) > \Essmin (\overline{D}) $, 
    \end{itemize}
where $\overline{\mathcal{D}}^{\mathrm{ad}} \in \widehat{\Div}(X)_\bR$ is the adelic $\bR$-Cartier divisor on $X$ induced by $\overline{\mathcal{D}} = (\mathcal{D}, (g_{\mathcal{D},v})_{v \in \Sigma_{K,\infty}})$.
\end{lemma}

\begin{proof} Since $\overline{D}$ is semi-positive, there exists a sequence $(\mathcal{X}_n, \mathcal{D}_n, (g_{n,v})_{v \in \Sigma_K})_{n \in \bN}$ such that~:
\begin{enumerate}
\item for all $n \in \bN$,  $(\mathcal{X}_n, \mathcal{D}_n)$ is a normal model for $(X,D)$ with $\mathcal{D}_n$ relatively nef, 
\item for all $n \in \bN$, $g_{n,v}$ is a smooth plurisubharmonic $D$-Green function invariant under $F_\infty$ if $v \in \Sigma_{K,\infty}$ and $g_{n,v}=g_{\mathcal{D}_n,v}$ for every non-archimedean $v \in \Sigma_K$,
\item\label{condlimit}  $\lim_{n \rightarrow +\infty} \|g_{n,v}-g_v\|_{v, \sup} = 0$ for every $v \in \Sigma_K$.
\end{enumerate}
 For each $n \in \bN$, we denote by $\pi_n \colon \mathcal{X}_n \rightarrow \mathcal{S}$ the morphism given by the model $\mathcal{X}_n$.  Let $M \in \Div(\mathcal{S})_\bR$ be an effective $\bR$-divisor such that $\widehat{\deg}(\overline{\xi}_M)=1$, where $\overline{\xi}_M$ is the adelic $\bR$-Cartier divisor on $\Spec K$ induced by $M$. Let $\varepsilon > 0$ be a real number. We let $\mathcal{D}_{n,\varepsilon}:= \mathcal{D}_n+\varepsilon \pi^*_nM$, $\overline{\mathcal{D}}_{n,\varepsilon} = (\mathcal{D}_{n,\varepsilon}, (g_{n,v})_{v \in \Sigma_{K,\infty}})$ and $\overline{\mathcal{D}}_{n} = (\mathcal{D}_{n}, (g_{n,v})_{v \in \Sigma_{K,\infty}})$.   
 
  By \eqref{condlimit} above and by  the projection formula \eqref{projformulaheight}, there exists an integer $n_{\varepsilon} \geq 1$ such that
 \[h_{\overline{\mathcal{D}}^{\mathrm{ad}}_{n,\varepsilon}}(x) = h_{\overline{\mathcal{D}}^{\mathrm{ad}}_n}(x) + \varepsilon \geq h_{\overline{D}}(x) + \varepsilon/2\] 
 for every $x \in X(\overline{K})$ and $n \geq n_{\varepsilon}$. It follows that $\Essmin(\overline{\mathcal{D}}^{\mathrm{ad}}_{n,\varepsilon})> \Essmin (\overline{D})$ for all $\varepsilon > 0$ and $n \geq n_{\varepsilon}$.  Moreover $\mathcal{D}_{n,\varepsilon}$ is relatively nef for all  $n \in \bN$ and $\varepsilon > 0$. It remains to check that we can choose  $\varepsilon$ and $n \geq n_{\varepsilon}$ so that $\overline{\mathcal{D}}_{n,\varepsilon}$ is not pseudo-effective. Since $\overline{D}$ is not pseudo-effective, there exists a big adelic $\bR$-divisor $\overline{B}$ on $X$ such that $\overline{D} + \overline{B}$ is not big. For all $n \in \bN$, $\varepsilon > 0$ and $v \in \Sigma_K$, we let $s_{v,n,\varepsilon} = \|g_{\mathcal{D}_{n,\varepsilon},v}-g_v\|_{v,\sup}$. By construction, we can choose $n$ and $\varepsilon$ so that $\max_{v \in \Sigma_K}s_{v,n,\varepsilon}$ is arbitrarily small. By continuity of $\widehat{\vol}$ (Theorem \ref{thmcontvol}), it follows that there exist  $\varepsilon_0>0$ and $m \geq n_{\varepsilon_0}$ such that $\overline{B}_{m, \varepsilon_0} := \overline{B} - (0, (s_{v,m,\varepsilon_0})_{v \in \Sigma_K}))$ is big. If $\overline{\mathcal{D}}_{m, \varepsilon_0}$ is pseudo-effective, then 
\[\widehat{\vol}(\overline{D}+\overline{B}) \geq  \widehat{\vol}(\overline{\mathcal{D}}^{\mathrm{ad}}_{m,\varepsilon_0} + \overline{B}_{m, \varepsilon_0})>0,\]
which is absurd. Hence $ \overline{\mathcal{D}}_{m, \varepsilon_0}$ is not pseudo-effective, and we obtain the lemma with $\overline{\mathcal{D}} = \overline{\mathcal{D}}_{m, \varepsilon_0}$.
\end{proof}

\subsection{The function field case}\label{paragffcase} The goal of this section is to prove Theorem \ref{thmmain} when $K$ is a function field, i.e.\ when $\mathcal{S} = C_K$ is a regular integral projective curve on a field $k$. We will see that in this case, Theorem \ref{thmmain} is a consequence of the following  theorem of Boucksom, Demailly, P\u{a}un and Perternell, describing the dual of the pseudo-effective cone of a projective variety. 
\begin{theorem}[Boucksom--Demailly--P\u{a}un--Peternell \cite{BDPP}]\label{thmBDPP}
 Let $Z$ be an integral projective variety over a field and let $\ell = \dim Z -1$. A divisor $E$ on $Z$ is pseudo-effective if and only if for every birational morphism $\varphi \colon Z' \rightarrow Z$ and for every ample divisors $A_1, \ldots , A_\ell$ on $\mathcal{X'}$, 
 \[E \cdot \varphi_*(A_1 \cdots A_\ell) \geq 0,\]
 where $\varphi_*$ denotes the proper pushforward of cycles. 
\end{theorem}

 In  \cite[Theorem 2.2]{BDPP},   this theorem is stated and proved over $\bC$. Recently, Das \cite{Das} observed that the proof given by Lazarsfeld in \cite[section 11.4.C]{LazII} (following \cite{BDPP}) works over an algebraically closed field of arbitrary characteristic. By the work of Cutkosky \cite{Cutkovsky}, it turns out that the proof given in \cite{LazII} remains valid without any assumption on the base field. More precisely, the key point in the proof of \cite[Theorem 11.4.19]{LazII} is the ``asymptotic orthogonality of Fujita approximation" \cite[Theorem 11.4.21]{LazII}. The latter holds for any field $k$ by \cite[Corollary 5.5]{Cutkovsky}\footnote{In \cite[Corollary 3.13]{Cutkovsky},  it follows immediately from the proof that one can replace the constant $C$ by $C'\omega^d$, where $C'$ is a numerical constant. This observation leads to a slightly more explicit statement in \cite[Corollary 5.5]{Cutkovsky}, which is exactly  \cite[Theorem 11.4.21]{LazII} for an arbitrary field $k$.}. The rest of the proof of \cite[Theorem 11.4.19]{LazII} makes no use of the algebraically closed nor characteristic zero assumption on $k$.

We will use the following classical consequence of Theorem \ref{thmBDPP} (see  \cite[Lemma 11.4.18]{LazII}).
\begin{coro}\label{coroBDPP} Let $Z$ be an integral projective variety over a field and let $E$ be a divisor  $Z$ which is not pseudo-effective. Then for any Zariski-closed proper subscheme $Y \subsetneq Z$, there exists an integral curve $C \subset Z$ not contained in $Y$ such that $E\cdot C <0$.
\end{coro}

We are now ready to prove Theorem \ref{thmmain} when $K$ is a function field. 

\begin{proof}[Proof of Theorem \ref{thmmain} ($K = k(T)$)] By Lemma \ref{lemmaineqessmin}, it suffices to prove the implication \[\Essmin(\overline{D}) \geq 0 \implies \overline{D} \text{ is pseudo-effective.}\]
 Assume that $\overline{D}$ is not pseudo-effective. By Lemma \ref{lemmaapproxdiv} (see also example \ref{examplepseffcdf}), there exists a normal model $(\mathcal{X},\pi_{\mathcal{X}})$ of $X$ on $C_K$ and a $\bR$-Cartier divisor $\mathcal{D}$ on $\mathcal{X}$ such that 
    \begin{enumerate}
    \item $\mathcal{D} $ is relatively nef,
    \item  $\mathcal{D} $ is not pseudo-effective,
    \item $ \Essmin(\overline{\mathcal{D}}^{\mathrm{ad}}) > \Essmin (\overline{D})$.
    \end{enumerate}
We only need to show that $\Essmin(\overline{\mathcal{D}}^{\mathrm{ad}}) \leq 0$.  Let $\mathcal{Y} \subsetneq \mathcal{X}$ be a proper closed subscheme. By Corollary \ref{coroBDPP}, there exists an integral curve $C_{\mathcal{Y}} \in \mathcal{X}$ not contained in $\mathcal{Y}$ such that $\mathcal{D} \cdot C_{\mathcal{Y}} < 0$. Since $\mathcal{D}$ is relatively nef, we have 
 $\mathcal{D}   \cdot C \geq 0$ for any curve contained in a fiber of $\pi_{\mathcal{X}}$ above a closed point. It follows that the curve $C_{\mathcal{Y}}$ is horizontal, in the sense that the restriction ${\pi_{\mathcal{X}}}_{|C_{\mathcal{Y}}}$ is a surjective morphism.  The generic point $P_{\mathcal{Y}} \in X(\overline{K})$ of $C_{\mathcal{Y}}$ satisfies 
\[h_{\overline{\mathcal{D}}^{\mathrm{ad}}}(P_{\mathcal{Y}}) = \frac{\mathcal{D}\cdot C_{\mathcal{Y}}}{[k(C_{\mathcal{Y}}) : K_0]} < 0.\]
 We infer that 
\[\sup_{\mathcal{Y} \varsubsetneq \mathcal{X}} \inf\{h_{\overline{\mathcal{D}}^{\mathrm{ad}}}(P) \ | \ P \in X(\overline{K}), \ C_P \nsubseteq \mathcal{Y}\} \leq 0,\]
where for any $P \in X(\overline{K})$, $C_P$ denotes the Zariski-closure of $P$ in $\mathcal{X}$. Now the supremum on the left hand-side is the essential minimum $\Essmin(\overline{\mathcal{D}}^{\mathrm{ad}})$, so the theorem is proved.
\end{proof}

\subsection{The number field case}\label{paragnfcase}

 We now assume that $K$ is a number field, i.e.\ $\mathcal{S} = \Spec \cO_K$. The following arithmetic analogue of Theorem \ref{thmBDPP} is due to Ikoma.

\begin{theorem}[\cite{Ikoma15}, Theorem 6.4]\label{thmIkoma} Let $\mathcal{X}$ be a normal projective arithmetic variety on $\Spec \cO_K$ of dimension $d+1$. Let $\overline{\mathcal{D}} = (\mathcal{D}, (g_v)_{v \in \Sigma_{K,\infty}})$ be an arithmetic $\bR$-divisor on $\mathcal{X}$. The following are equivalent. 
\begin{enumerate}
\item $\overline{\mathcal{D}}$ is pseudo-effective.
\item For any blowing-up $\varphi \colon \mathcal{X}' \rightarrow \mathcal{X}$ such that $\mathcal{X}'$ is normal and generically smooth, and for any ample arithmetic $\bQ$-divisor  $\overline{\mathcal{H}}$ on $\mathcal{X}'$, we have
\[\widehat{\deg}(\varphi^*\overline{\mathcal{D}} \cdot \overline{\mathcal{H}}^d) \geq 0.\]
\end{enumerate}
\end{theorem}

We shall derive Theorem \ref{thmmain} from Ikoma's theorem by applying a method similar to the one we used in the function field case.  In order to deduce an analogue of Corollary \ref{coroBDPP} from Theorem \ref{thmIkoma}, we use an arithmetic Bertini-type theorem due to Moriwaki \cite[Theorem 5.3]{Moriwaki94}. The latter was extended by Ikoma \cite{IkomaBertini}, and more recently by Charles \cite{Charles}.

\begin{proof}[Proof of Theorem \ref{thmmain}] By Lemma \ref{lemmaineqessmin}, it is sufficient to assume that $\overline{D}$ is not pseudo-effective and to show that $\Essmin(\overline{D}) < 0$. By Lemma \ref{lemmaapproxdiv}, there exists a normal projective arithmetic variety $\mathcal{X}$ on $\Spec \cO_K$ and an arithmetic $\bR$-Cartier divisor $\overline{\mathcal{D}} = (\mathcal{D},(g_{v})_{v \in \Sigma_{K,\infty}})$ such that~:
\begin{enumerate} 
    \item $\overline{\mathcal{D}} $ is semi-positive of $C^\infty$-type,
    \item $\overline{\mathcal{D}} $ is not pseudo-effective,
    \item $ \Essmin(\overline{\mathcal{D}}^{\mathrm{ad}}) > \Essmin (\overline{D}) $.
\end{enumerate}
It is enough to show that $\Essmin(\overline{\mathcal{D}}^{\mathrm{ad}}) \leq 0$. By linearity and continuity of the essential minimum (Lemma \ref{propertiesEssmin} \eqref{contEssmin}), we may assume that $\overline{\mathcal{D}}^{\mathrm{ad}} \in \widehat{\Div}(\mathcal{X})_{\bZ}$.  Since $\overline{\mathcal{D}} $ is not pseudo-effective,  by Theorem \ref{thmIkoma} there exists a birational morphism $\varphi \colon \mathcal{X}' \rightarrow \mathcal{X}$, with $\mathcal{X}'$ normal and generically smooth, and an ample arithmetic $\bQ$-divisor $\overline{\mathcal{H}}$ on $\mathcal{X}'$ such that
\[\widehat{\deg}(\varphi^*\overline{\mathcal{D}} \cdot \overline{\mathcal{H}}^d) <0,\]
where $d = \dim X$. Without loss of generality, we assume that $\mathcal{H} \in \Div(\mathcal{X})_{\bZ}$. Let $\mathcal{Y} \subsetneq \mathcal{X}$ be a closed subscheme. Since $\overline{\mathcal{H}}$ is ample, \cite[Theorem A]{IkomaBertini} implies the existence of an integer $n\geq 1$  and a non-zero function $\phi \in H(\mathcal{X}',n\mathcal{H})$ with $\max_{v \in \Sigma_{K,\infty}}\|\phi\|_{v,\sup} <1$  such that the divisor $(\phi)+n\mathcal{H}$ is generically smooth and intersects $\mathcal{Y}':=\varphi^{-1}(\mathcal{Y})$ properly. We write
\[(\phi) + n\mathcal{H}= \sum_{i=1}^\ell a_i \mathcal{Z}_i\]
the decomposition of $(\phi) +n\mathcal{H}$ as a cycle~: for each $i \in \{1,\ldots, \ell\}$, $a_i > 0$ is an integer and $\mathcal{Z}_i \subset \mathcal{X}'$ is a smooth subvariety of codimension $1$ not contained in $\mathcal{Y}'$. We have 
\begin{multline}\label{ineqintersection}
\widehat{\deg}(\varphi^*\overline{\mathcal{D}} \cdot (n\overline{\mathcal{H}})^d)  = \sum_{i=1}^\ell a_i\widehat{\deg}(\overline{\mathcal{D}}_{|\mathcal{Z}_i} \cdot (n\overline{\mathcal{H}}_{|\mathcal{Z}_i})^{d-1})\\ -\sum_{v \in \Sigma_{K,\infty}}\int_{X_v^{\mathrm{an}}} \ln \|\phi\|_v c_{1,v}(\varphi^*\overline{\mathcal{D}})\wedge c_{1,v}(n\overline{\mathcal{H}})^{d-1}.
\end{multline}
Since $\varphi^*\overline{\mathcal{D}}$ is semi-positive and $\overline{\mathcal{H}}$ is ample,  the current $c_{1,v}(\varphi^*\overline{\mathcal{D}})\wedge c_{1,v}(n\overline{\mathcal{H}})^{d-1}$ is non-negative for each $v \in \Sigma_{K,\infty}$.  On the other hand $\|\phi\|_{v, \sup} \leq 1$, hence we have 
\[ -\int_{X_v^{\mathrm{an}}} \ln \|\phi\|_v c_{1,v}(\varphi^*\overline{\mathcal{D}})\wedge c_{1,v}(\overline{n\mathcal{H}})^{d-1} \geq 0\]
for each $v \in \Sigma_{K,\infty}$. It follows from \eqref{ineqintersection} that
\[ \sum_{i=1}^\ell a_i\widehat{\deg}(\overline{\mathcal{D}}_{|\mathcal{Z}_i} \cdot (n\overline{\mathcal{H}}_{|\mathcal{Z}_i})^{d-1}) \leq   \widehat{\deg}(\varphi^*\overline{\mathcal{D}} \cdot (n\overline{\mathcal{H}})^d)  = n^{d}\widehat{\deg}(\varphi^*\overline{\mathcal{D}} \cdot \overline{\mathcal{H}}^d) < 0.\]
Since $a_i > 0$ for all $1 \leq i \leq \ell$, there exists a codimension $1$ generically smooth subvariety $\mathcal{Z} \subset \mathcal{X}'$ not contained in $\mathcal{Y}'$ such that 
\[\widehat{\deg}(\varphi^*\overline{\mathcal{D}}_{|\mathcal{Z}} \cdot \overline{\mathcal{H}}_{|\mathcal{Z}}^{d-1})<0.\]
Since $\overline{\mathcal{H}}_{|\mathcal{Z}}$ is ample, we can iterate this process. Eventually we obtain an integral curve $\mathcal{C}' \subset \mathcal{X}'$ not contained in $\mathcal{Y}'$ such that 
\[\widehat{\deg}(\varphi^*\overline{\mathcal{D}}_{|\mathcal{C}'}) < 0.\]
 We denote by $\pi_{\mathcal{X}} \colon \mathcal{X} \rightarrow \mathcal{S} = \Spec \cO_K$ the morphism defining the model $\mathcal{X}$, and we let $\pi_{\mathcal{X}'} =  \pi_{\mathcal{X}}\circ\varphi$. Assume that $\pi_{\mathcal{X}'}(\mathcal{C}') \in \mathcal{S}$ is a closed point, and let $p$ be its image in $\Spec \bZ$. In this case we have 
\[\widehat{\deg}(\varphi^*\overline{\mathcal{D}}_{|\mathcal{C}'}) = \ln(p)\deg(\varphi^*\mathcal{D} \cdot \mathcal{C}') < 0,\]
 which is absurd since $\varphi^*\mathcal{D}$ is relatively nef. It follows that $\mathcal{C}'$ is horizontal, i.e.\ $\pi_{\mathcal{X}'}(\mathcal{C}') = \mathcal{S}$. The image $\varphi(\eta_{\mathcal{C}'})$ of the generic point of $\mathcal{C}'$ by $\varphi$ defines a closed point $x \in X(\overline{K})\setminus \mathcal{Y}$ of height $h_{\overline{\mathcal{D}}^{\mathrm{ad}}}(x) < 0$. It follows that 
 \[\inf \{h_{\overline{\mathcal{D}}^{\mathrm{ad}}}(x) \ | \ x \in X(\overline{K}),\ x \notin \mathcal{Y}\} < 0.\]
 By taking the supremum on the Zariski-closed proper subschemes $\mathcal{Y} \subset \mathcal{X}$ we find $\Essmin(\overline{\mathcal{D}}^{\mathrm{ad}}) \leq 0$, which concludes the proof. 
\end{proof}

\section{Asymptotic slopes and proof of Conjecture \ref{conjChen}}\label{sectionasyslopes}
We recall definitions and known facts about asymptotic slopes in subsection \ref{sectiondefasy}.  These invariants were first introduced by Chen in \cite{Chenpolygone} for arithmetic divisors (see also \cite{Chenfujita}), and extended to adelic $\bR$-divisors by Chen and Moriwaki \cite{ChenMoriwaki, CMadeliccurves}. In subsection \ref{sectionproofconj} we prove Theorem \ref{thmmain2intro}. Let $K$ be a global field  and let $\pi \colon X \rightarrow \Spec K$ be a projective, normal and geometrically integral variety of dimension $d := \dim X$.  We fix a place $v_0$ of $K_0$, and we assume for simplicity that $v_0 = \infty$ if $K_0 = \bQ$.

\subsection{Definitions and properties}\label{sectiondefasy}
 
  Let $\overline{D} \in \widehat{\Div}(X)_\bR$.  For any real number $t$ and any integer $n \geq 1$, we consider the $K$-linear subspace $V_n^{t}$  of $V_n =  H^0(X,nD)$ generated by functions $\phi \in V_n$ such that for each place $v \in \Sigma_K$,
 \[\|\phi\|_{v,\sup} \leq \left\{\begin{tabular}{ll}
 $e^{-t}$ & if $v | v_0$,\\
 $1$ & otherwise.
 \end{tabular}\right.\]
  We let 
\[\lambda_{\max,n}(\overline{D}) := \sup \{t \in \bR \ | \ V_n^{t} \ne 0\}
\text{ and }
\lambda_{\min,n}(\overline{D}) := \sup \{t \in \bR \ | \ V_n^{t} = V_n\}.\]

\begin{example} (1) Suppose that $K$ is  a number field and that $\overline{D} = \overline{\mathcal{D}}^{\mathrm{ad}}$ for some arithmetic $\bR$-Cartier divisor $\overline{\mathcal{D}}$ on a normal model $\mathcal{X}$ of $X$. Then for any $n \geq 1$ and $t \in \bR$, $V_n^{t}$ identifies with the set 
\[\{\phi \in H^0(\mathcal{X}, n \mathcal{D}) \ | \ \max_{v \in \Sigma_{K,\infty}}\|\phi\|_{v, \sup} \leq e^{-t}\}.\]

(2) Suppose that $K$ is a number field and let $n \geq 1$ be an integer. Then $\exp(-\lambda_{\max,n}(\overline{D}))$ and $\exp(-\lambda_{\min,n}(\overline{D}))$ are respectively the first and last minima of $(V_n, (\|\|_{v,\sup})_{v \in\Sigma_K})$ in the sense of Bombieri--Vaaler \cite[section III.1]{BombieriVaaler}. 
\end{example}

\begin{defi}\label{defiasyslopes} The \textit{asymptotic maximal slope} of $\overline{D}$ is the quantity
\[\pam (\overline{D}) = \limsup_{n \rightarrow +\infty} \frac{\lambda_{\max,n}(\overline{D})}{n}.\]
The \textit{asymptotic minimal slope} of $\overline{D}$ is
\[\pim (\overline{D}) = \liminf_{n \rightarrow +\infty} \frac{\lambda_{\min,n}(\overline{D})}{n}.\]
\end{defi}

\begin{rema}\label{remadefpam}  For any integer $n \geq 1$ the $K$-vector space $V_n$ equipped with the collections of norms $(\|.\|_{v,\sup})_{v \in \Sigma_K}$ is an adelic $K$-vector space in the sense of Gaudron \cite{Gaudronpentes}, and we let $\pmax(\overline{V}_n)$ be its maximal slope as defined in \cite[section 5.1]{Gaudronpentes}. By   \cite[Theorem 5.20]{Gaudronpentes}, we have $ \pmax(\overline{V}_n) = \lambda_{\max,n}(\overline{D}) + o(n)$.    Therefore
\[\pam (\overline{D}) = \limsup_{n \rightarrow +\infty} \frac{\pmax(\overline{V}_n)}{n}.\]
In particular, Definition \ref{defiasyslopes} of $\pam (\overline{D})$  is consistent with the one of \cite{Chenpolygone}. The same is true for the asymptotic minimal slope.  
\end{rema}

The asymptotic slopes of $\overline{D}$ are actually limits under suitable positivity conditions on $D$.

\begin{prop}[\cite{Chenpolygone}, Theorem 4.1.2]\label{proppamconverge} If $D$ is big, the sequence $(\lambda_{\max,n}(\overline{D})/n)_{n \in \bN}$ converges in $\bR$~: 
$\pam (\overline{D}) = \lim_{n \rightarrow +\infty} \lambda_{\max,n}(\overline{D})/n$.
If moreover the graded $K$-algebra 
\[V_\bullet = \bigoplus_{n \in \bN} V_n\] is finitely generated,  the sequence $(\lambda_{\min,n}(\overline{D})/n)_{n \in \bN}$ converges in $\bR$~: $\pim (\overline{D}) = \lim_{n \rightarrow +\infty} \lambda_{\min,n}(\overline{D})/n$.
\end{prop}

  The following result of Chen \cite[Proposition 3.11]{Chenfujita} characterizes the positivity of the asymptotic maximal slope in terms of bigness of $\overline{D}$. It holds in the more general setting of adelic curves by Chen--Moriwaki \cite{CMadeliccurves}.
\begin{prop}[\cite{CMadeliccurves}, Proposition 6.4.18]\label{proppamsup} 
The following conditions are equivalent~:
\begin{enumerate}
\item $\overline{D}$ is big,
\item $D$ is big and $\pam(\overline{D}) > 0$.
\end{enumerate}
\end{prop}

\subsection{Proof of Theorem \ref{thmmain2intro}}\label{sectionproofconj} The goal of this subsection is to prove the following theorem, which is a slightly more precise version of Theorem \ref{thmmain2intro} in the introduction. We say that a divisor $D \in \Div(X)_\bQ$ is semi-ample if there exists an integer $n \geq 1$ such that $nD \in \Div(X)_\bZ$ and $\cO_X(nD)$ is generated by global sections.

\begin{theorem}\label{thmpamessmin}  Let $\overline{D} \in \widehat{\Div}(X)_\bR$. 
\begin{enumerate}
\item\label{itempamessmin} Suppose that $D$ is big. Then
\[\Essmin(\overline{D}) \geq \pam(\overline{D}),\]
 with equality if $\overline{D}$ is semi-positive.
\item\label{itempimessmin} Suppose that $D \in \Div(X)_\bQ$ is big and semi-ample. Then
\[\zeta_{\mathrm{abs}}(\overline{D}) \geq  \pim(\overline{D}),\]
 with equality if $\overline{D}$ is semi-positive.
\end{enumerate}
\end{theorem}

Note that the first assertion of this theorem implies Conjecture \ref{conjChen}. As we will see, it is an immediate consequence of Corollary \ref{coromain} and  Proposition \ref{proppamsup}. We shall derive Theorem \ref{thmpamessmin} \eqref{itempimessmin} from an arithmetic Nakai-Moishezon criterion for ampleness, originally proved by Zhang when $K$ is a number field \cite[Theorem 4.2]{Zhangplav} and extended to the much wider setting of adelic curves by Chen and Moriwaki  \cite{CMadeliccurves}, under weaker positivity assumptions.

\begin{proof} \eqref{itempamessmin} Assume that $D$ is big. Let $\overline{\xi}_0 = (\xi_{0,v})_{v \in \Sigma_K}$ be the adelic divisor on $\Spec K$ given by 
\[\xi_{0,v} = \left\{\begin{tabular}{ll}
$2$ & if $v | v_0$,\\
$0$ & otherwise.
\end{tabular} \right.\] 
 It follows from the definitions that $\pam(\overline{D} - t\pi^*\overline{\xi}_0) = \pam(\overline{D}) -t$ for any $t \in \bR$. Applying  Proposition \ref{proppamsup} to $\pam(\overline{D} - t\pi^*\overline{\xi}_0)$ when $t \in \bR$ varies, we find
\begin{equation*}
\begin{split}
\pam(\overline{D}) & = \sup\{t \in \bR \ | \ \overline{D} - t\pi^*\overline{\xi}_0  \text{ is big}\}\\
& = \sup\{t \in \bR \ | \ \overline{D} - t\pi^*\overline{\xi}_0  \text{ is pseudo-effective}\},
\end{split}
\end{equation*}
where the second equality follows from the definition of pseudo-effectivity. We conclude with Corollary \ref{coromain} (note that $\widehat{\deg}(\overline{\xi}_0)=1$).

\eqref{itempimessmin} For $t \in \bR$, we let $\overline{D}_t := \overline{D} - t\pi^*\overline{\xi}_0$. It follows from the projection formula for the height \eqref{projformulaheight} and the definition of $\zeta_{\mathrm{abs}}(\overline{D})$ that 
\[\zeta_{\mathrm{abs}}(\overline{D}) = \sup \{t \in \bR \ | \ \overline{D}_t  \text{ is nef}\}.\] 
Let $n \geq 1$ be an integer and let $(\phi_1, \ldots, \phi_{\ell_n})$ be a basis for $H^0(X,nD)$. Let $x \in X(\overline{K})$ be a point with residue field $K(x)$. Since $D$ is semi-ample, there exists $i \in \{1, \ldots, \ell_n\}$ such that $x\notin \Supp (D + (\phi_i))$ provided that $n$ is sufficiently large and divisible. It follows that
\[nh_{\overline{D}}(x) = -\sum_{w \in \Sigma_{K(x)}}n_w(K(x)) \ln\|\phi_i\|_w(x) \geq -\sum_{w \in \Sigma_{K(x)}}n_w(K(x)) \ln\|\phi_i\|_{w,\sup}.\]
We infer that $h_{\overline{D}}(x) \geq \lambda_{\min, n}(\overline{D})/n$ for every $x \in X(\overline{K})$ and $n \gg 1$ sufficiently divisible. By the projection formula \eqref{projformulaheight}, $D_t = \overline{D} -t\pi^*\overline{\xi}_0$ is nef for every $t \leq \pim(\overline{D})$, so that 
\[\pim(\overline{D}) \leq \sup\{t \in \bR \ | \ \overline{D}_t \text{ is nef}\} = \zeta_{\mathrm{abs}}(\overline{D}).\] 
We now assume that $\overline{D}$ is semi-positive to prove the converse inequality.  By definition, $\pim(\overline{D}_t) = \pim(\overline{D})-t$ for  any $t \in \bR$.  By the above, there exists $t_0 \in \bR$ such that $\overline{D}_t$ is nef for every $t \leq t_0$. Let $t < t_0$ be a rational number and let $Y \subset X$ be an irreducible subvariety. We have
 \[\pam({\overline{D}_t}_{|Y}) = \Essmin({\overline{D}_t}_{|Y}) \geq \Essabs({\overline{D}_t}_{|Y}) > \Essabs(\overline{D}_{t_0}) \geq 0.\]
By definition of $\pam({\overline{D}_t}_{|Y})$ there exists an integer $n_Y\geq 1$ such that $\widehat{H}^0(Y,n{\overline{D}_t}_{|Y}) \ne 0$ for every $n \geq n_Y$.  By the arithmetic Nakai-Moishezon criterion \cite[Theorem 7.4.1]{CMadeliccurves},  we have $\pim(\overline{D})-t = \pim(\overline{D}_t) \geq 0$, and finally
\[\pim(\overline{D}) \geq \sup\{t \in \bR \ | \ \overline{D}_t  \text{ is nef}\} = \zeta_{\mathrm{abs}}(\overline{D}).\]
\end{proof}

\begin{rema}\label{remapamcursor} When $D$ is big, we proved that 
\begin{equation}\label{eqpamcursor}
\pam(\overline{D})  = \sup\{t \in \bR \ | \ \overline{D} - t\pi^*\overline{\xi}_0  \text{ is pseudo-effective}\},
\end{equation}
where $\overline{\xi}_0 = (\xi_{0,v})_{v \in \Sigma_K} \in \widehat{\Div}(\Spec K)_\bR$ is defined by $\xi_{0,v} = 2$ if $v |v_0$ and $\xi_{0,v} = 0$ otherwise. This equality actually holds for any   $\overline{\xi} \in \widehat{\Div}(\Spec K)_\bR$ with $\widehat{\deg}(\overline{\xi}) = 1$ (see \cite{ChenMoriwaki}). Note that \eqref{eqpamcursor} shows that Theorem \ref{thmmain} is actually equivalent to Conjecture \ref{conjChen}. 
\end{rema}

\section{The Boucksom--Chen concave transform}\label{sectionBCtransform}
The purpose of this section is to recall the definition and some known properties of the concave transform introduced by Boucksom and Chen \cite{BoucksomChen}. We closely follow Moriwaki \cite[section 1.1]{MoriwakiMZ2012}, who extended this construction to the case of $\bR$-Cartier divisors.

\subsection{Okounkov bodies}\label{sectionOkounkov} 
 We briefly recall the construction of an Okounkov body attached to a graded linear series of a divisor,  following   \cite{LazMus} and \cite{MoriwakiMZ2012}.   Let $K$ be a global field and $X \rightarrow \Spec K$ be a projective, normal and geometrically integral variety of dimension $d := \dim X$. Let $D$ be a divisor on $X$, and let $D_{\overline{K}}$ be the pullback of $D$ to $X_{\overline{K}} = X\times_K \Spec \overline{K}$. The choice of a system of parameters $\mathbf{z}$ centred at a regular point $P \in X(\overline{K})$ defines a function
\[\nu_{\mathbf{z}} \colon H^0(X_{\overline{K}},D_{\overline{K}})\setminus \{0\} \rightarrow \bR^d\]
satisfying the following conditions (see \cite[section 1]{MoriwakiMZ2012} for details)~:
\begin{itemize}
 \item for every $a \in F^\times$ and $\phi \in H^0(X_{\overline{K}},D_{\overline{K}})\setminus \{0\}$, $\nu_{\mathbf{z}}(a\phi) = \nu_{\mathbf{z}}(\phi)$;
 \item  for all $\phi_1,\phi_2 \in H^0(X_{\overline{K}},D_{\overline{K}})\setminus \{0\}$, 
 \[\nu_{\mathbf{z}}(\phi_1\phi_2) = \nu_{\mathbf{z}}(\phi_1) + \nu_{\mathbf{z}}(\phi_2) \text{ and } \nu_{\mathbf{z}}(\phi_1 +\phi_2) \geq_{\mathrm{lex}} \min\{\nu_{\mathbf{z}}(\phi_1),\nu_{\mathbf{z}}(\phi_2)\},\]
 \end{itemize}
 where $\geq_{\mathrm{lex}}$ denotes the lexicographic order on $\bR^d$. 
  Given a graded $K$-subalgebra $W_\bullet = \bigoplus_{n \in \bN}W_n$ of $ V_\bullet(D):= \bigoplus_{n \in \bN} H^0(X,nD)$, we consider the subset of $\bR^d$ defined by
 \[\Gamma_{\mathbf{z}}(W_\bullet) =  \bigcup_{n \geq 1} \frac{1}{n}\nu_{\mathbf{z}}(W_n \otimes_K \overline{K} \setminus \{0\}) = \left\{ \frac{\nu_{\mathbf{z}}(\phi)}{n} \ | \ n \geq 1, \ 0 \ne \phi \in W_n \otimes_K \overline{K} \right\}.\]
 
 \begin{defi}
 The \textit{Okounkov body of $W_\bullet$} with respect to the system of parameters $\mathbf{z}$ is defined to be the closure \[\Delta_{\mathbf{z}}(W_\bullet) = \overline{\Gamma_{\mathbf{z}}(W_\bullet)}\] of $\Gamma_{\mathbf{z}}(W_\bullet)$ in $\bR^d$ for the euclidean topology. We call  $\Delta_{\mathbf{z}}(D) = \Delta_{\mathbf{z}}(V_\bullet(D))$ the \textit{Okounkov body of $D$} with respect to $\mathbf{z}$.
\end{defi} 
 When $D$ is big, $\Delta_{\mathbf{z}}(D) \subset \bR^d$ is a convex body (see \cite[Proposition 2.1]{LazMus}), and we have $\vol(D) = d!\vol(\Delta_{\mathbf{z}}(D))$.

 \subsection{The concave transform}
Let $\overline{D} = (D,(g_v)_{v \in \Sigma_K})$ be  an adelic $\bR$-Cartier divisor  on $X$ with $D$ big.  For any  $t \in \bR$ and $n\in \bN$, let $V_n^t$ be the $K$-linear subspace of $V_n = H^0(X,nD)$ defined in section \ref{sectionasyslopes}.  \begin{defi}  The \textit{concave transform} of $\overline{D}$ with respect to $\mathbf{z}$ is the function $G_{\overline{D},\mathbf{z}}\colon \Delta_{\mathbf{z}}(D) \rightarrow \bR \cup \{-\infty\}$ defined by
 \[G_{\overline{D},\mathbf{z}}(\alpha) = \sup \{t \ | \ \alpha \in \Delta_{\mathbf{z}}(V_\bullet^t)\}\]
 for every $\alpha \in \Delta_{\mathbf{z}}(D)$.
 \end{defi}
By \cite[section 1.3]{BoucksomChen}, $G_{\overline{D},\mathbf{z}}$ is an upper-semicontinuous concave function and we have
\[\pim(\overline{D}) \leq G_{\overline{D},\mathbf{z}}(\alpha) \leq \pam(\overline{D})\]
for any $\alpha \in  \Delta_{\mathbf{z}}(D)$. The maximal asymptotic slope  actually coincides with the maximum of $G_{\overline{D},\mathbf{z}}$. 
\begin{lemma}\label{lemmapamcf}
 Assume that $D$ is big. Then we have 
 \[\pam(\overline{D}) = \max_{\alpha \in \Delta_{\mathbf{z}}(D)}G_{\overline{D},\mathbf{z}}(\alpha) .\]
\end{lemma}

 \begin{proof} We have already seen that $G_{\overline{D},\mathbf{z}}(\alpha) \leq \pam(\overline{D})$
 for any $\alpha \in  \Delta_{\mathbf{z}}(D)$.  For the reverse inequality, let $\overline{\xi}_0  \in \widehat{\Div}(\Spec K)_\bR$ be the $\bR$-adelic divisor defined in remark \ref{remapamcursor}. Let $t \in \bR$ be such that $\overline{D} - t\pi^*\overline{\xi}_0$ is big. For $n \geq 1$ sufficiently large, there exists a non-zero  $\phi \in \widehat{H}^0(X,n\overline{D}-nt\pi^*\overline{\xi}_0) \subset V^{nt}_n$. It follows that $\alpha_\phi := \frac{1}{n}\nu_{\mathbf{z}}(s) \in  \Delta_{\mathbf{z}}(V_\bullet^t)$, hence $ G_{\overline{D},\mathbf{z}}(\alpha_\phi) \geq t$. Therefore
 \begin{equation*}
 \begin{split}
 G_{\overline{D},\mathbf{z}}(\alpha_\phi)& \geq \sup\{t \in \bR \ | \ \overline{D} - t\pi^*\overline{\xi}_0  \text{ is big}\}\\
  & =  \sup\{t \in \bR \ | \ \overline{D} - t\pi^*\overline{\xi}_0  \text{ is pseudo-effective}\}=\pam(\overline{D}),
 \end{split}
 \end{equation*}
 so that $\max_{\alpha \in \Delta_{\mathbf{z}}(D)}G_{\overline{D},\mathbf{z}}(\alpha)\geq \pam(\overline{D})$.
\end{proof}

 \subsection{Concave transform and volume functions}\label{sectionconctransfvol} Theorem \ref{thmXiVOlOkounkov} below gives two useful formulae due to  Boucksom and Chen \cite{BoucksomChen}, expressing the volumes $\widehat{\vol}(\overline{D})$ and  $\widehat{\vol}_{\chi}(\overline{D})$ in terms of the concave transform. We first recall the definition  of $\widehat{\vol}_{\chi}(\overline{D})$ for an adelic $\bR$-Cartier divisor $\overline{D}$ on $X$. Let $\mathbb{A}_K$ be the ring of ad\`{e}les of $K$. For any integer $n \geq 1$, we let $V_n= H^0(X,nD)$ and we consider the adelic unit ball $\mathbb{B}(V_n) \subset V_n\otimes_K  \mathbb{A}_K$ defined by
\[\mathbb{B}(V_n) = \{(\phi_v) \in V_n \otimes_K \mathbb{A}_K\ | \ \|\phi_v\|_v \leq 1 \ \forall v \in \Sigma_K\}.\]
We denote by $\nu$ the Haar measure on $V_n \otimes_K \mathbb{A}_K$ with $\nu((V_n \otimes_K \mathbb{A}_K)/V_n)=1$. 
We define the $\chi$\textit{-volume} of $\overline{D}$ by
\[\widehat{\vol}_{\chi}(\overline{D}) = \limsup_{n \rightarrow +\infty} \frac{\nu(\mathbb{B}(V_n))}{n^{d+1}/(d+1)!}.\]
\begin{theorem}[Boucksom--Chen \cite{BoucksomChen}]\label{thmXiVOlOkounkov}
Let $\overline{D}$ be an adelic $\bR$-Cartier divisor on $X$ such that $D$ is big. 
 Then 
 \[\widehat{\vol}(\overline{D}) = (d+1)![K:K_0] \int_{\Delta_{\mathbf{z}}(D)} \max\{0,G_{\overline{D},\mathbf{z}}\} d\lambda\]
  and
\[\widehat{\vol}_{\chi}(\overline{D}) \leq (d+1)! [K:K_0]\int_{\Delta_{\mathbf{z}}(D)} G_{\overline{D},\mathbf{z}} d\lambda,\] 
with equality if $\inf_{\alpha \in \Delta_{\mathbf{z}}(D)}  G_{\overline{D},\mathbf{z}}(\alpha) > -\infty$.
\end{theorem}
\begin{proof} The result  follows from \cite[Theorem 1.11]{BoucksomChen} as in \cite[Theorems 2.8 and 3.1]{BoucksomChen}  (see also \cite[Theorem 1.2.1]{MoriwakiMZ2012}). In the function field case, one should replace Gillet-Soulé's result used in the proof of \cite[Theorem 2.8]{BoucksomChen} by its geometric analogue, see for example \cite[Theorem 2.4]{ChenMathZ}. Alternatively, one can see Theorem \ref{thmXiVOlOkounkov} as a corollary of \cite[Theorem 6.4.9]{CMadeliccurves}, which is a much more general statement valid for adelic curves.
\end{proof}

\section{Application to Zhang's theorem}\label{sectionZhang}

Let $K$ be a global field and  let $X \rightarrow \Spec K$ be a projective, normal and geometrically integral variety of dimension $d := \dim X \geq 1$.  Let $\overline{D} = (D,(g_v)_{v \in \Sigma_K})$ be an adelic $\bR$-Cartier divisor on $X$.  Our goal in this section is to prove Theorem \ref{thmineqZhangintro} in the introduction. We will derive it from Boucksom--Chen's theorem \ref{thmXiVOlOkounkov} and the following proposition, which follows immediately from Theorem \ref{thmpamessmin} \eqref{itempamessmin} and Lemma \ref{lemmapamcf}.  We fix a choice of parameters $\mathbf{z}$ centred at a regular point $P \in X(\overline{K})$.
 \begin{prop}\label{propEssminmaxcf}
 If $D$ is big, then
\[\Essmin(\overline{D}) \geq \max_{\alpha \in \Delta_{\mathbf{z}}(D)}G_{\overline{D},\mathbf{z}}(\alpha) = \pam(\overline{D}),\]
 with equality if  $\overline{D}$ is semi-positive. 
 \end{prop}

We begin with a variant of Theorem \ref{thmineqZhangintro} involving the arithmetic volumes $\widehat{\vol}(\overline{D})$ and $\widehat{\vol}_{\chi}(\overline{D})$. 
\begin{theorem}\label{thmineqZhang} 
 Assume that $D$ is big. 
\begin{enumerate}
\item\label{thmineqZhang1} We have 
\[(d+1)[K:K_0]\Essmin(\overline{D}) \geq \frac{\widehat{\vol}_{\chi}(\overline{D})}{\vol(D)},\]
with equality if and only if $G_{\overline{D},\mathbf{z}}(\alpha) = \Essmin(\overline{D})$ for every $\alpha \in  \Delta_{\mathbf{z}}(D)$. 
\item\label{thmineqZhang2} Assume that $\overline{D}$ is pseudo-effective. Then 
\[(d+1)[K:K_0]\Essmin(\overline{D}) \geq \frac{\widehat{\vol}(\overline{D})}{\vol(D)},\]
with equality if and only if $\max \{0,G_{\overline{D},\mathbf{z}}(\alpha)\} = \Essmin(\overline{D})$ for every $\alpha \in  \Delta_{\mathbf{z}}(D)$. 
\end{enumerate}
\end{theorem}

In the special case when $\overline{D}$ is a toric metrized divisor on a toric variety, this theorem was proved by Burgos Gil, Philippon and Sombra \cite{BPSmin}. One warning~:  the definition of the volumes $\widehat{\vol}(\overline{D})$ and $\widehat{\vol}_{\chi}(\overline{D})$ used in \cite{BPSmin} are normalized by the degree $[K:K_0]$.

\begin{proof} Assume that $D$ is big. By Proposition \ref{propEssminmaxcf} and Theorem \ref{thmXiVOlOkounkov}, we have 
\begin{equation*}
\begin{split}
\Essmin(\overline{D})  \geq \max_{\alpha \in \Delta_{\mathbf{z}}(D)}G_{\overline{D},\mathbf{z}}(\alpha)&  \geq \frac{d!}{\vol(D)} \int_{\Delta_{\mathbf{z}}(D)} G_{\overline{D},\mathbf{z}} d\lambda\\
& \geq \frac{\widehat{\vol}_{\chi}(\overline{D})}{(d+1)[K:K_0]\vol(D)}.
\end{split}
\end{equation*}
If $(d+1)[K:K_0]\Essmin(\overline{D}) = \frac{\widehat{\vol}_{\chi}(\overline{D})}{\vol(D)}$, then by the above we have
\[\Essmin(\overline{D}) = \max_{\alpha \in \Delta_{\mathbf{z}}(D)}G_{\overline{D},\mathbf{z}}(\alpha) = \frac{1}{\vol(\Delta_{\mathbf{z}}(D))}\int_{\Delta_{\mathbf{z}}(D)} G_{\overline{D},\mathbf{z}} d\lambda,\]
so $G_{\overline{D},\mathbf{z}}=\Essmin(\overline{D})$ is constant. Conversely, if $G_{\overline{D},\mathbf{z}}$ is constant equal to $\Essmin(\overline{D})$ then applying Proposition \ref{propEssminmaxcf} and Theorem \ref{thmXiVOlOkounkov} again we obtain 
\[\frac{\widehat{\vol}_{\chi}(\overline{D})}{[K:K_0]\vol(D)} = \frac{(d+1)!}{\vol(D)}\vol(\Delta_{\mathbf{z}}(D)) \max_{\alpha \in \Delta_{\mathbf{z}}(D)}G_{\overline{D},\mathbf{z}} (\alpha) = (d+1)\Essmin(\overline{D}).\]

 If  $\overline{D}$ is pseudo-effective, then $\max_{\alpha \in \Delta_{\mathbf{z}}(D)}G_{\overline{D},\mathbf{z}}(\alpha)  \geq 0$ by remark \ref{remapamcursor} and Lemma \ref{lemmapamcf}. The proof of \eqref{thmineqZhang2} is the exact analogue of \eqref{thmineqZhang1} when the function $G_{\overline{D},\mathbf{z}}$ is replaced by $ \max\{0,G_{\overline{D},\mathbf{z}}\}$.
\end{proof}

We will deduce Theorem \ref{thmineqZhangintro} from Theorem \ref{thmineqZhang} by comparing the $\chi$-volume $\widehat{\vol}_{\chi}(\overline{D})$ with the   \textit{height} $h_{\overline{D}}(X)$ of $X$. We refer the reader to \cite[section 2.5]{BPS14} for the definition of the height $h_{\overline{D}}(X)$ of $X$. When $K$ is a number field and $\overline{D}$ is integrable (see section \ref{sectionintersection}), then $[K:\bQ]h_{\overline{D}}(X) = \widehat{\deg}(\overline{D}^{d+1})$. If $K$ is a function field and the $D$-Green functions of $\overline{D}$ are all induced by a divisor $\mathcal{D} \in \Div(\mathcal{X})_{\bR}$ on a normal model $\mathcal{X}$ of $X$, then 
$[K:k(T)]h_{\overline{D}}(X) = \mathcal{D}^{d+1}$
is the top self-intersection number of $\mathcal{D}$. When $\overline{D}$ is semi-positive, it turns out that  $h_{\overline{D}}(X)$ coincides with $\widehat{\vol}_{\chi}(\overline{D})$ up to normalization, by the following theorem of Moriwaki \cite{MoriwakiMAMS}. 
\begin{theorem}[Moriwaki]\label{thmvolchiMoriwaki} If $\overline{D}$ is semi-positive, then $\widehat{\vol}_{\chi}(\overline{D}) = [K:K_0]h_{\overline{D}}(X)$. 
\end{theorem}

\begin{proof} In the number field case, this result is proved in \cite[proof of Theorem 5.3.2]{MoriwakiMAMS}. We briefly explain how to adapt the proof of \cite{MoriwakiMAMS} when $K$ is a a function field. By a continuity argument as in \cite{MoriwakiMAMS}, we reduce the problem to the case where all the $D$-Green functions of $\overline{D}$ are induced by a normal model $(\mathcal{X}, \mathcal{D})$ of $(X,D)$ with $\mathcal{D} \in \Div(\mathcal{X})_\bZ$ relatively ample with respect to the corresponding morphism $\pi \colon \mathcal{X} \rightarrow C_K$. Let $\mathcal{L} = \cO_{\mathcal{X}}(\mathcal{D})$. In this case,
\[[K:k(T)]h_{\overline{D}}(X) = \mathcal{L}^{d+1}\  \text{ and }  \ \widehat{\vol}_{\chi}(\overline{D}) = \limsup_{n \rightarrow + \infty} \frac{\deg(\pi_*(\mathcal{L}^{\otimes n}))}{n^{d+1}/(d+1)!},\]
where $\deg(\pi_*(\mathcal{L}^{\otimes n}))$ is the geometric degree of the vector bundle $\pi_*(\mathcal{L}^{\otimes n}))$ on $C_K$. By the Riemann-Roch formula, we have 
\[\deg(\pi_*(\mathcal{L}^{\otimes n})) = h^0(C_K,\pi_*(\mathcal{L}^{\otimes n})) - h^1(C_K,\pi_*(\mathcal{L}^{\otimes n})) + O(n^d).\]
 Since $\mathcal{L}$ is relatively ample, we have  $H^i(C_K,\pi_*(\mathcal{L}^{\otimes n})) = H^i(\mathcal{X},\mathcal{L}^{\otimes n})$ for all $i \geq 0$ and $n \gg 1$. In particular, $H^i(\mathcal{X},\mathcal{L}^{\otimes n}) = 0$ for all $n \gg 1$ and $i > 1 = \dim C_K$. We conclude that $\widehat{\vol}_{\chi}(\overline{D}) =\mathcal{L}^d$ by the asymptotic Riemann--Roch theorem.
\end{proof}
Combining  Proposition \ref{propEssminmaxcf} with Theorems \ref{thmineqZhang} and \ref{thmvolchiMoriwaki}, we obtain Theorem \ref{thmineqZhangintro} in the introduction. We reproduce the statement below.  

\begin{coro}\label{coroineqZhangheight} If $\overline{D}$ is semi-positive and $D$ is big, then
\[\Essmin(\overline{D}) \geq \frac{h_{\overline{D}}(X) }{(d+1)D^{d}},\]
with  equality if and only if the following equivalent conditions are satisfied~:
\begin{enumerate}
\item $G_{\overline{D},\mathbf{z}}$ is constant;
\item\label{condeqZ2} the sequence $(\lambda_{\max,n}(\overline{D})/n)_{n \geq 1}$ converges to $\frac{h_{\overline{D}}(X) }{(d+1)D^{d}}$.
\end{enumerate}
In that case, $G_{\overline{D},\mathbf{z}}(\alpha) = \Essmin(\overline{D})$ for any $\alpha \in \Delta_{\mathbf{z}}(\overline{D})$.
\end{coro}

\begin{proof} Assume that $D$ is big and $\overline{D}$ is semi-positive. In particular $D$ is nef and big, so that $D^d = \vol(D)> 0$. By Proposition \ref{propEssminmaxcf} and Theorem \ref{thmineqZhang}, the three following conditions are equivalent~:
\begin{enumerate}
\item $G_{\overline{D},\mathbf{z}}$ is constant equal to $\Essmin(\overline{D})$;
\item  $G_{\overline{D},\mathbf{z}}$ is constant;
\item $\lim_{n\rightarrow + \infty}(\lambda_{\max,n}(\overline{D})/n)={\widehat{\vol}_{\chi}(\overline{D})}/([K:K_0](d+1)D^{d})$.
\end{enumerate}
By Theorem  \ref{thmvolchiMoriwaki}, $\widehat{\vol}_{\chi}(\overline{D}) = [K:K_0]h_{\overline{D}}(X)$. Therefore the result follows from Theorem \ref{thmineqZhang}.
\end{proof}

\begin{rema}\label{remacritasystable} Theorem \ref{thmpamessmin} \eqref{itempamessmin} gives another approach to prove  criterion \eqref{condeqZ2} in Corollary \ref{coroineqZhangheight}, based on the theory of adelic vector spaces of Gaudron \cite{Gaudronpentes} and without using the concave transform. Assume that $D$ is big. We keep the notations of remark \ref{remadefpam}, so that for any $n \geq 1$, $\overline{V}_n = (V_n, (\|.\|_{v,\sup})_{v \in \Sigma_K})$ denotes the adelic vector space \cite{Gaudronpentes} given by the $K$-vector space $V_n = H^0(X,nD)$ equipped with the supremum norms. Since $D$ is big, $V_n \ne \{0\}$ for $n$ large enough.  Let $\widehat{\deg}(\overline{V}_n)$ be the \textit{normalized adelic degree} of $\overline{V}_n$ as defined in \cite[Definition 4.1]{Gaudronpentes} and let $\widehat{\mu}(\overline{V}_n) := \widehat{\deg}(\overline{V}_n)/\dim V_n$ be its \textit{slope}.   It follows from the definitions that
\begin{equation}\label{eqvolchipentes}
\frac{\widehat{\vol}_{\chi}(\overline{D})}{[K:K_0]\vol(D)} = (d+1)\limsup_{n \rightarrow +\infty} \frac{\widehat{\deg}(\overline{V}_n)}{n\dim(V_n)} = (d+1)\limsup_{n \rightarrow +\infty}\frac{\widehat{\mu}(\overline{V}_n)}{n}.
\end{equation}
On the other hand, we saw in remark \ref{remadefpam} that 
$\pam(\overline{D}) = \lim_{n \rightarrow +\infty}\pmax(\overline{V}_n)/n$,
where $\pmax(\overline{V}_n)$ denotes the maximal slope of $\overline{V}_n$ \cite[Definition 5.4]{Gaudronpentes}, which by definition satisfies $\pmax(\overline{V}_n) \geq \widehat{\mu}(\overline{V}_n)$ for any $n \geq 1$. It immediately follows from Theorem \ref{thmpamessmin} and \eqref{eqvolchipentes} that
\[(d+1)\Essmin(\overline{D}) \geq (d+1)\lim_{n \rightarrow +\infty}\frac{\pmax(\overline{V}_n)}{n} \geq \frac{\widehat{\vol}_{\chi}(\overline{D})}{[K:K_0]\vol(D)}.\]
If $\overline{D}$ is semi-positive, then the first inequality is actually an equality by Theorem \ref{thmpamessmin} and $\widehat{\vol}_{\chi}(\overline{D}) = [K:K_0]h_{\overline{D}}(X)$ by Theorem \ref{thmvolchiMoriwaki}. Therefore we have the equivalence
\[(d+1)\Essmin(\overline{D})= \frac{h_{\overline{D}}(X)}{\vol(D)} \Longleftrightarrow \lim_{n \rightarrow +\infty}\frac{\pmax(\overline{V}_n)}{n} = \limsup_{n \rightarrow +\infty}\frac{\widehat{\mu}(\overline{V}_n)}{n}.\]
 The equality on the right-hand side can be interpreted as an ``asymptotic semi-stability" condition (see \cite[section 5.3]{Gaudronpentes} for the classical notion of semi-stability). This gives another interesting criterion for equality to hold in Zhang's theorem \ref{thmZhangintro}. We will investigate this approach further in the next two sections, restricting our attention to projective spaces.
\end{rema}

\section{Hermitian vector spaces}\label{sectionhermitian}

Let $K$ be a global field and let $\overline{K}$ be an algebraic closure of $K$. In this section we recall the definitions of hermitian vector spaces on $K$, following Gaudron \cite{Gaudronpentes, Gaudronart18} and Gaudron--Rémond \cite{GR13}. Successive minima and slopes are defined in subsections \ref{sectionheightdegreevs} and \ref{paragslopes}.  

\subsection{Definitions}\label{sectionhermitianvsdefi}
Given a place $v \in \Sigma_K$ and an integer $d \geq 1$, we fix an algebraic closure $\bC_v$ of the completion $K_v$ of $K$ at $v$.  We define a norm $\|.\|_{2,v}$ on $\bC_v^d$ by
\[\forall \ \mathbf{x} = (x_1, \ldots, x_d) \in \bC_v^d, \ \ \|\mathbf{x}\|_{2,v} = \left\{\begin{tabular}{ll}
$\left(\sum_{i = 1}^d |x_i|_v^2\right)^{1/2}$ & if $v | \infty$,\\
$\max_{1 \leq i \leq d} |x_i|_v$ & if $v \nmid \infty$
\end{tabular} \right. .\]
We denote by $\mathbb{A}_K$ the ring of ad\`{e}les of $K$.

\begin{defi}\label{defihermitianvs} A \textit{hermitian $K$-vector space} is the data $\overline{E} = (E, (\|.\|_{\overline{E},v})_{v \in \Sigma_K})$ of a finite dimensional $K$-vector space $E$ and for each place $v \in \Sigma_K$, a norm $\|.\|_{\overline{E},v}$ on $E \otimes_K \bC_v$ such that the following condition holds. There exist a $K$-basis $(e_1, \ldots, e_d)$ of $E$ and an adelic matrix $A = (A_v)_{v \in \Sigma_K} \in \GL_d(\mathbb{A}_K)$ such that for any $v \in \Sigma_K$ and $\mathbf{x} = (x_1, \ldots, x_d) \in \bC_v^d$, we have
\[\|x_1e_1+ \cdots + x_de_d\|_{\overline{E},v} = \|A_v \mathbf{x}\|_{2,v}.\]
A \textit{hermitian $\overline{K}$-vector space} is a hermitian vector space defined on some finite extension of $K_0=\bQ$ or $k(T)$.
\end{defi}

\begin{rema} If $K'$ is an algebraic extension of $\bQ$, the notion of $K'$-hermitian vector space coincides with the one of \textit{rigid adelic space} on $K'$ introduced by Gaudron and Rémond \cite{GRSiegel}. 
\end{rema}

Let $\overline{E} = (E,(\|.\|_{\overline{E},v})_{v \in \Sigma_K})$ and $\overline{E}' = (E',(\|.\|_{\overline{E}',v})_{v \in \Sigma_K})$ be two hermitian vector spaces on $K$ of dimension $d,d'\geq 1$ respectively. Given a place $v \in \Sigma_K$, we say that a basis $(f_1, \ldots, f_d)$ of $E\otimes_K \bC_v$ is \textit{orthonormal} if 
\[\|x_1f_1 + \cdots + x_df_d\|_{\overline{E},v} = \|\mathbf{x}\|_{2,v}\]
for any $\mathbf{x} = (x_1, \ldots, x_d) \in \bC_v^d$. By definition, such a basis exists for any $v \in \Sigma_K$.

\subsection{Operations}
We briefly recall some basic operations on hermitian vector spaces. We refer the reader to \cite[section 3.3]{Gaudronpentes} for details.
\subsubsection{Subspace and quotient} We say that $ \overline{E}' \subset \overline{E}$ is a hermitian subspace of $\overline{E}$ if $E' \subset E$ and if for every place $v \in \Sigma_K$, $\|.\|_{\overline{E}',v}$ is the restriction of $\|.\|_{\overline{E},v}$ to $E' \otimes_K \bC_v$. We define the quotient hermitian space $\overline{E}/\overline{E}'$ by considering the quotient norms.

\subsubsection{Dual space}\label{sectionDual} The dual $\overline{E}^\vee = (E^\vee,(\|.\|_{\overline{E}^\vee,v})_{v \in \Sigma_K})$ of $\overline{E}$ is the hermitian vector space given by the dual vector space $E^\vee = \Hom_K(E,K)$ equipped with the usual dual norms on $E^\vee \otimes_K \bC_v$.  The bidual ${(\overline{E}^{\vee})}^{\vee}$ is isometrically isomorphic to $\overline{E}$ by \cite[Remark 3.6]{Gaudronpentes}. This observation will be important in the sequel.

\subsubsection{Tensor product and symmetric power}\label{sectionSym} For each $v \in \Sigma_K$, we define a norm $\|.\|_{\overline{E}\otimes \overline{E}',v}$ on $E\otimes_K E' \otimes_K \bC_v$ as follows.  Let $(e_1, \ldots, e_d)$ and $(e'_1, \ldots, e'_{d'})$ be orthonormal bases for $E \otimes_K \bC_v$ and $E'\otimes_K \bC_v$. Then $\|.\|_{\overline{E}\otimes \overline{E}',v}$ is defined to be the unique norm such that the basis $(e_i \otimes e'_j)_{1 \leq i \leq d, 1 \leq j \leq d'}$ is orthonormal. The hermitian tensor product is the hermitian vector space $\overline{E}\otimes_K \overline{E}' =(E \otimes_K E', (\|.\|_{\overline{E}\otimes \overline{E}',v})_{v \in \Sigma_K})$. For any integer $n \geq 1$, the symmetric power $S^n\overline{E}$ is the hermitian vector space $(S^nE, (\|.\|_{S^n\overline{E},v})_{v \in \Sigma_K})$ given by the quotient hermitian structure of $\overline{E}^{\otimes n}$.

\subsubsection{Wedge product} Let $r \in \{1, \ldots, d\}$. We define the $r$-th wedge product $\wedge^r \overline{E}$ of $\overline{E}$ to be the hermitian vector space $(\wedge^r E, (\|.\|_{\wedge^r\overline{E},v})_{v \in \Sigma_K})$, where for each $v \in \Sigma_K$, $\|.\|_{\wedge^r\overline{E},v}$ is the norm on $\wedge^rE  \otimes_K \bC_v$ such that
\[\|\eta\|_{\wedge^r\overline{E},v} = \inf \{ \|x_1\|_{\overline{E},v} \cdots \|x_r\|_{\overline{E},v} \ | \  x_1, \ldots, x_r \in E \otimes_K \bC_v, \  \eta = x_1 \wedge \cdots \wedge x_d \}\]
for all $\eta \in \wedge^rE  \otimes_K \bC_v$.
 One warning~: as opposed to the symmetric power, the wedge product is not exactly the quotient of $\overline{E}^{\otimes r}$ (see \cite[section 2.7]{GR13}). The determinant $\det \overline{E}$ of $\overline{E}$ is by definition the hermitian vector space $\det \overline{E} = \wedge^d \overline{E}$.

\subsubsection{Scalar extension} For any finite extension $K'$ of $K$ and for any place $w \in \Sigma_{K'}$ above a place $v \in \Sigma_K$, we define a norm $\|.\|_{w}$ on $E \otimes_{K} \bC_v$ as follows~: if $\sigma_v \colon K \hookrightarrow \bC_v$ and $\sigma_w \colon K' \hookrightarrow \bC_v$ are embeddings associated to $v$ and $w$ respectively, we let
\[\left\| \sum_i (e_i \otimes_K x_i) \otimes_K^{\sigma_w} y_i \right\|_w = \left\| \sum_i e_i \otimes_K^{\sigma_v}( \sigma_w(x_i)y_i)\right\|_v \]
for any finite families of elements $e_i \in E$, $x_i \in K'$, $y_i \in \bC_v$. This gives to $E_{K'} = E\otimes_K K '$ the structure of an hermitian vector space on $K'$, denoted by $\overline{E}_{K'} = (E_{K'}, (\|.\|_{w})_{w \in \Sigma_{K'}})$.

Any $\overline{K}$-vector subspace $F$ of $E_{\overline{K}} = E \otimes_K \overline{K}$ is a $K'$-vector subspace of $E_{K'} = E \otimes_K K'$ for some finite extension $K'$ of $K$. Hence $F$ has a natural structure of $\overline{K}$-hermitian vector space $\overline{F} \subset \overline{E}_{K'}$.

\subsection{Height function and successive minima}\label{sectionheightdegreevs} Let $\overline{E}=(E, (\|.\|_v)_{v \in \Sigma_K})$ be $K$-hermitian of dimension $d \geq 1$.  For any non-zero vector $s \in E \otimes_K \overline{K}$, we define the \textit{height} of $s$ by
\[h_{\overline{E}}(s) = \sum_{w \in \Sigma_{K'}} n_w(K') \ln \|s\|_w,\]
where $K'$ is a finite extension of $K$ such that $s \in E\otimes_K K'$ (see section \ref{sectionconventions} for the definition of $n_w(K')$). This definition does not depend on the choice of the field $K'$. Moreover, we have $h_{\overline{E}}(\lambda s)=h_{\overline{E}}(s)$ for any $\lambda \in \overline{K}^\times$ by the product formula \eqref{productformula}.

 \begin{defi} Let $i \in \{1, \ldots, d\}$. For all $\lambda \in \bR$, we consider the set
 \[E(\lambda, \overline{K}) = \{s \in E\otimes_K \overline{K}\setminus \{0\} \ | \ h_{\overline{E}}(s) \leq \lambda\}.\]
  The \textit{Zhang $i$-th minimum} of $\overline{E}$ is  defined by
 \[\zeta_i(\overline{E}) = \inf \{\lambda \in \bR \ | \ \dim \mathrm{Zar}(E(\lambda, \overline{K})) \geq i\},\] 
 where $\mathrm{Zar}(E(\lambda, \overline{K}))$ denotes the Zariski closure of $E(\lambda, \overline{K})$ in $E \otimes_K \overline{K}$ (for any choice of basis $E \otimes_K \overline{K} \simeq \overline{K}^d$).
  \end{defi}
  
  This terminology is justified by the following remark.
  \begin{rema}\label{remaZmin} We denote by $\bP_K(E^\vee) = \mathrm{Proj}_K (\mathrm{Sym} E^\vee)$ the projective space associated to $E^\vee$. Let $\phi \in E^\vee$ be a non-zero vector  and let $D = (\phi) \in \Div(\bP_K(E^\vee))$. For any $v \in \Sigma_K$, we consider the $D$-Green function \[g_v \colon \bP_{K_v}(E^\vee \otimes_K K_v) \setminus \Supp(D_v) \rightarrow \bR\] defined by $g_v(x) = 2\ln \|x_\phi\|_v$, where $x_\phi \in E \otimes_K K_v$ is the unique representative of $x\in \bP_{K_v}(E^\vee \otimes_K K_v)$ with $\phi(x) =1$.  We denote by $\overline{D}$  the adelic Cartier divisor $(D, (g_v)_{v \in \Sigma_K})$. It follows from the definitions that the Zhang minima of $\overline{E}$ coincide with the successive minima of $\overline{D}$ defined in subsection \ref{sectionZmin}, namely
  $\zeta_i(\overline{E}) = \zeta_i(\overline{D})$
  for any $i \in\{1, \ldots, d\}$.
  \end{rema}
 
 \subsection{Successive slopes}\label{paragslopes}
 Let $\overline{E}=(E, (\|.\|_v)_{v \in \Sigma_K})$ be a hermitian vector space on $K$ of dimension $d\geq 1$. 
 \begin{defi}\label{defidegree} The \textit{degree} of $\overline{E}$ is the quantity 
\[\widehat{\deg} (\overline{E}) = -h_{\det \overline{E}}(s),\]
 where $s \in \det (E) \otimes_K \overline{K}$ is any non-zero vector. We put $\widehat{\deg}(\{0\}) = 0$.  The \textit{slope} of $\overline{E}$ is defined by $\widehat{\mu}(\overline{E}) = \widehat{\deg}(\overline{E})/\dim(E)$.
\end{defi}
 Note that by the product formula, this definition does not depend on the choice of $s \in \det E \setminus \{0\}$. Moreover, $\widehat{\deg} (\overline{E}_{K'})=\widehat{\deg} (\overline{E})$ for any finite extension $K'$ of $K$. In particular, Definition \ref{defidegree} naturally extends to $\overline{K}$-hermitian vector spaces.  Let $F \subset E_{\overline{K}}$ be a $\overline{K}$-vector subspace  of dimension $r \geq 1$. The degree of $\overline{F}$ is given by
\[\widehat{\deg}(\overline{F}) = - h_{\wedge^{r}\overline{E}}(\eta_F),\]
where $\eta_F$ is any non-zero vector in $\det (F) \subset \wedge^rE $. In particular, we have $\widehat{\deg}(\overline{F}) \leq -\zeta_1(\wedge^r \overline{E}) < \infty$. It follows that the  convex hull of the set 
\[\{(\dim F, \widehat{\deg}\overline{F}) \ | \ F \subset E_{\overline{K}}\}\subset \bR^2\]
is delimited from above by a concave and piecewise affine function $P_{\overline{E}} \colon [0,d] \rightarrow \bR$. 
\begin{defi}\label{defimui} For any $i \in \{1, \ldots , d\}$, the $i$-th slope of $\overline{E}$ is defined by 
$\widehat{\mu}_i(\overline{E}) = P_{\overline{E}}(i) - P_{\overline{E}}(i-1)$.
\end{defi}
Note that $ \widehat{\deg} (\overline{E}) =\sum_{i=1}^{d} \widehat{\mu}_i (\overline{E})$
and
 $\widehat{\mu}_{d} (\overline{E}) \leq \cdots \leq \widehat{\mu}_1 (\overline{E})$
by concavity of $P_{\overline{E}}$. We recall now some classical properties of the successive slopes. 
\begin{prop}\label{propmui} 
 For all $i \in \{1,\ldots, d\}$,  we have
 \begin{enumerate}
 \item\label{propmui2} $\widehat{\mu}_{i}(\overline{E}) =-\widehat{\mu}_{d -i+1} (\overline{E}^{\vee})$;
 \item\label{propmui3} $\widehat{\mu}_{i}(\overline{E}) =   \max_{E_1} \min_{E_2} \widehat{\mu}(\overline{E}_1/\overline{E}_2) = \min_{E_2} \max_{E_1}  \widehat{\mu}(\overline{E}_1/\overline{E}_2)$, 
 where $E_2 \subset E_1$ run over the subspaces of $E_{\overline{K}}$ with $\dim E_2 < i \leq \dim E_1$.
\end{enumerate}  
\end{prop}
\begin{proof}
When $K$ is a number field, this is \cite[Propositions 18 and 19 (2)]{Gaudronart18}. The proofs remain valid when $K$ is a function field.
\end{proof}
We define the \textit{maximal slope} of $\overline{E}$ by $\pmax(\overline{E}) = \widehat{\mu}_1(\overline{E})$, and the \textit{minimal slope} by $\pmin(\overline{E})= \widehat{\mu}_{d}(\overline{E})$. By Proposition \ref{propmui} \eqref{propmui3}, we have 
\begin{equation*}
\begin{split}
\pmax(\overline{E}) &=  \max \{\widehat{\mu}(\overline{F}) \ | \ 0 \ne F \subset E_{\overline{K}}\},\\
  \pmin(\overline{E}) &=  \min \{\widehat{\mu}(\overline{G}) \ | \  0 \ne G \text{ quotient of } E_{\overline{K}}\}.
\end{split}
\end{equation*}
\begin{rema} When $K$ is a number field, our definition of successive slopes coincides with the ones of \cite{Gaudronpentes}, as one can see by using a classical Galois descent argument (see for example \cite[Proposition 19 (3)]{Gaudronart18}). In particular, we have
\[\pmax(\overline{E}) =  \max \{\widehat{\mu}(\overline{F}) \ | \ 0 \ne F \subset E\} \ \text{ and } \  \pmin(\overline{E}) =  \min \{\widehat{\mu}(\overline{E}/\overline{F}) \ | \  F\subsetneqq E\}.\]
\end{rema}

In the next section we will compare successive slopes and minima of hermitian vector spaces. For this purpose we need the following proposition due to Gaudron.
\begin{prop}\label{propcomp0} For every $i \in \{1, \ldots, d\}$, we have
\[0 \leq \widehat{\mu}_i(\overline{E})+ \zeta_i(\overline{E}) \leq \sup_{\dim E' = d}\left( \widehat{\mu}_{d}(\overline{E}')+ \zeta_{d}(\overline{E}')\right),\]
where the supremum is over all $\overline{K}$-hermitian vector spaces $\overline{E}'$ with $\dim E' = d$.
\end{prop}

\begin{proof} In the number field case, this is \cite[Corollary 23 and Proposition 29]{Gaudronart18}. The proofs remain valid when $K$ is a function field. 
\end{proof}

\section{Proof of Theorem \ref{thmminslopesintro} and applications}\label{sectionGN}

In this section we work over a number field $K$. We shall explain how to adapt the arguments for function fields  in subsection \ref{paragcdf}. 

\subsection{Proof of Theorem \ref{thmminslopesintro}}\label{sectionproofminslopes}

Let $\overline{E}$ be a hermitian $K$-vector space of dimension $d \geq 1$. We want to prove the equalities
 \[\zeta_d(\overline{E}) = \lim_{n \rightarrow +\infty} \frac{\pmax(S^n (\overline{E}^\vee))}{n}
\ \text{ and } \   \zeta_1(\overline{E}) = \lim_{n \rightarrow +\infty} \frac{\pmin(S^n (\overline{E}^\vee))}{n}.\]
To do so we apply Theorem \ref{thmpamessmin} to the divisor $\overline{D}$ of remark \ref{remaZmin}. We then compare the quantities $\lambda_{\max,n}(\overline{D})$, $\lambda_{\min,n}(\overline{D})$ with the maximal and minimal slopes of $S^n (\overline{E}^\vee)$, closely following Chen \cite[section 4.2]{Cheniccm}. The reader more familiar with the formalism of adelic line bundles can replace $\overline{D}$ by $\cO_{\bP(E^\vee)}(1)$ equipped with Fubiny-Study norms coming from $\overline{E}$.
 
  Let $\overline{D}$ be the adelic Cartier divisor on $X = \bP_K(E^\vee)$ constructed in remark \ref{remaZmin}. In particular, we have $\zeta_d(\overline{E}) = \Essmin(\overline{D})$,  $\zeta_1(\overline{E}) = \zeta_{\mathrm{abs}}(\overline{D})$ and $V_n:= H^0(X,nD) = S^n E^\vee$ for all $n \geq 1$. Moreover $\overline{D}$ is semi-positive and the underlying divisor $D$ is ample. By Theorem \ref{thmpamessmin}, we have
 \begin{equation}\label{eqminslopes}
\zeta_d(\overline{E}) = \pam(\overline{D}) \ \text{ and } \   \zeta_1(\overline{E}) = \pim(\overline{D}).
\end{equation}

  We can equip the $K$-vector space $V_n =  S^n E^\vee$ with two different families of norms $(\|.\|_{1,v})_v$, $(\|.\|_{2,v})_v$ defined as follows~: for each $v \in \Sigma_K$,
 \begin{itemize}
 \item $\|.\|_{1,v}=\|.\|_{v,\sup}$ is the supremum norm associated to $\overline{D}$,
 \item $\|.\|_{2,v}=\|.\|_{S^n(\overline{E}^\vee),v}$ is the norm induced by $\overline{E}$ constructed in subsections \ref{sectionDual} and \ref{sectionSym}.
 \end{itemize}
For any $i \in \{1,2\}$ and $t \in \bR$, we denote by $V_{i,n}^t$ the $K$-linear subspace of $V_n$ generated by vectors $\phi \in V_n$ such that $\|\phi\|_{i,v}\leq e^{-t}$ for all $v | v_0$ and $\|\phi\|_{i,v}\leq 1$ for all $v \nmid v_0$. We let 
\[\lambda_{\max}(V_n, \|.\|_i) = \sup \{t \in \bR \ | \ V_{i,n}^t \ne \{0\}\}, \ \ \lambda_{\min}(V_n, \|.\|_i) = \sup \{t \in \bR \ | \ V_{i,n}^t = V_n\}.\]
By definition, we have 
\[\pam(\overline{D}) = \lim_{n \rightarrow +\infty} \frac{\lambda_{\max}(V_n, \|.\|_1)}{n} \ \text{ and } \   \pim(\overline{D}) =\lim_{n \rightarrow +\infty} \frac{\lambda_{\min}(V_n,\|.\|_1)}{n}
\]
(the limits exist by Proposition \ref{proppamconverge}).
Let $v \in \Sigma_K$ be a place. We let $\delta(v) = 1 $ if $v$ is archimedean and $\delta(v)=0$ otherwise. By  the classical inequalities of norms (\cite[Lemma 7.6]{Gaudronpentes})
\begin{equation*}\label{ineqBGS}
\|.\|_{1,v} \leq \|.\|_{2,v} \leq \binom{n+d-1}{d-1}^{\delta(v)/2}\|.\|_{1,v},
\end{equation*}
we have \[\lambda_{\max}(V_n, \|.\|_1) =\lambda_{\max}(V_n, \|.\|_2) +o(n), \ \ \lambda_{\min}(V_n, \|.\|_1) = \lambda_{\min}(V_n, \|.\|_2) + o(n).\]
 On the other hand, 
\[\lambda_{\max}(V_n, \|.\|_2) =\pmax(S^n(\overline{E}^\vee)) +o(n)\  \text{ and } \ \lambda_{\min}(V_n, \|.\|_2) = \pmin(S^n(\overline{E}^\vee))+ o(n)\]
by \cite[Theorem 1.1]{Chencompare}. It follows that 
 \[\pam(\overline{D}) = \lim_{n \rightarrow +\infty} \frac{\pmax(S^n (\overline{E}^\vee))}{n}
\ \text{ and } \   \pim(\overline{D})) = \lim_{n \rightarrow +\infty} \frac{\pmin(S^n (\overline{E}^\vee))}{n},\]
and we conclude with \eqref{eqminslopes}.

\subsection{Symmetry defects}

Let $\overline{E}$ be a hermitian $K$-vector space of dimension $d \geq 1$. In view of Theorem \ref{thmminslopesintro}, we introduce two invariants controlling the behaviour of the maximal slope with respect to symmetric products.

\begin{defi}\label{defsymdefect} The \textit{strong symmetry defect of} $\overline{E}$ is the quantity
\[\alpha_{\mathrm{s}}(\overline{E}) = \lim_{n \rightarrow +\infty} \frac{1}{n}\left(\widehat{\mu}_{\max}(S^n(\overline{E}^\vee)) - \widehat{\mu}_{\max}((S^n\overline{E})^\vee)\right),\]
and the \textit{symmetry defect of} $\overline{E}$ is defined by
\[\alpha(\overline{E}) = \lim_{n \rightarrow +\infty} \frac{1}{n}\left(\pmax(S^n(\overline{E}^\vee)) - n\pmax(\overline{E}^\vee)\right).\]
Let $\alpha_{\sharp}$ denote either  $\alpha_s$ or $\alpha$. For any integer $d \geq 1$, we define 
\[\alpha_{\sharp}(d) = \sup \{\alpha_{\sharp}(\overline{F}) \ | \ \overline{F} \text{ hermitian } \overline{K}\text{-vector space of dimension } \leq d\}\]
and we put $\alpha_{\sharp}(0)=0$.
\end{defi}

For any integer $N \in \bN$, we denote by $H_N = 1 + 1/2 + \ldots + 1/N$ the $N$-th harmonic number (with the convention $H_0=0$).

\begin{prop}\label{propsymdefect} For any positive integer $d$, we have
\[\ln(d)/2 \leq \alpha(d) \leq \alpha_{\mathrm{s}}(d) \leq H_{d - 1}.\]
\end{prop}

\begin{proof} The first inequality follows from the identity $\alpha((\bQ^d, (\|.\|_{2,v})_v)) = \ln(d)/2 $ (see \cite[page 585]{GR13} for an explicit calculation).  Let $\overline{E}$ be a hermitian vector bundle of dimension $d\geq 1$ on $K$. 
 By Proposition \ref{propcomp0} and Theorem \ref{thmminslopesintro}, we have
\[\pmin(\overline{E}) = -\widehat{\mu}_1(\overline{E}^\vee) \leq \zeta_1(\overline{E}^\vee) = \lim_{n\rightarrow +\infty}\frac{\widehat{\mu}_{\min}(S^n\overline{E})}{n}.\]
Therefore, 
\begin{equation*}
\begin{split}
0\leq \alpha(\overline{E}) &  = \lim_{n \rightarrow +\infty} \frac{1}{n}\left(\widehat{\mu}_{\max}(S^n(\overline{E}^\vee)) + n\widehat{\mu}_{\min}(\overline{E})\right)\\
& \leq  \lim_{n \rightarrow +\infty} \frac{1}{n}\left(\widehat{\mu}_{\max}(S^n(\overline{E}^\vee)) + \widehat{\mu}_{\min}(S^n\overline{E})\right) = \alpha_s(\overline{E}).
\end{split}
\end{equation*}
 To give an upper bound for $\alpha_s(\overline{E})$, we follow \cite[Proof of Lemma 6.2]{GR13}. For all $n \in \bN$, we consider the isomorphism $\theta_n \colon S^n(E^\vee) \rightarrow (S^nE)^\vee$
defined as follows~: for $\varphi_1\cdots \varphi_n \in S^n(E^\vee)$ and $x_1 \cdots x_n \in S^n E$, we let
\[\theta_n(\varphi_1\cdots \varphi_n)(x_1\cdots x_n) = \sum_{\sigma \in \mathfrak{S}_n}\prod_{i=1}^n \varphi_i(x_{\sigma(i)}),\]
where $\mathfrak{S}_n$ denotes the symmetric group of $\{1, \ldots, n\}$. By \cite[Lemma 6.4]{Gaudronpentes}, we have 
\[\widehat{\mu}_{\max}(S^n(\overline{E}^\vee)) \leq \widehat{\mu}_{\max}((S^n\overline{E})^\vee) + \sum_{v \in \Sigma_K}n_v(K) \ln\|\theta_n\|_v,\]
where $\|\theta_n\|_v$ denotes the operator norm of the induced map $\theta_n \colon S^n(E^\vee) \otimes_K \bC_v \rightarrow (S^nE)^\vee\otimes_K \bC_v$ for each $v \in \Sigma_K$. By \cite[pages 583 and 589]{GR13}, the sum on the right hand side is of the form $nH_{d-1} + o(n)$. The result follows.
\end{proof}

\subsection{Comparison of slopes and minima}
 Let $\overline{E}$ be a hermitian adelic vector space on $K$ of dimension $d \geq 1$. As an application of Theorem \ref{thmminslopesintro}, we now give upper bounds for the quantities $\widehat{\mu}_i(\overline{E}) + \zeta_i(\overline{E})$ in terms of symmetry defects. 
\begin{coro}\label{corocompmin} For any $ i \in \{1,\ldots, d\}$, we have
\[\widehat{\mu}_i(\overline{E}) + \zeta_i(\overline{E})  \leq \alpha(d) \leq H_{d-1}.\]
\end{coro}

\begin{proof}  It is enough to consider the case $i=d$ by Proposition \ref{propcomp0}, and we have
\[\widehat{\mu}_d(\overline{E}) + \zeta_d(\overline{E}) = \widehat{\mu}_d(\overline{E}) + \lim_{n \rightarrow +\infty} \frac{\pmax(S^n(\overline{E}^\vee))}{n} = \alpha(\overline{E})\leq \alpha(d) \leq H_{d-1}\]
by Theorem \ref{thmminslopesintro} and Proposition \ref{propsymdefect}.
\end{proof}

\subsection{An absolute transference theorem}

 Let $\overline{E}$ be a hermitian $K$-vector space of dimension $d \geq 1$. The following theorem gives upper bounds for the transference problem in terms of symmetry defects. By Proposition \ref{propsymdefect}, it implies Theorem \ref{thmtransferenceintro} in the introduction. 

\begin{theorem}\label{thmtransference}
For any $i \in\{1, \ldots d\}$, we have 
 \[\zeta_i(\overline{E}) + \zeta_{d-i+1}(\overline{E}^\vee) \leq \min\{2\alpha(d),\alpha_{\mathrm{s}}(i) + \alpha_{\mathrm{s}}(d-i+1) \}.\]
\end{theorem}

\begin{proof}  Let $i \in \{1, \ldots d\}$. By Proposition \ref{propmui} and Corollary \ref{corocompmin}, we have
\[\zeta_i(\overline{E}) + \zeta_{d-i+1}(\overline{E}^\vee) \leq - \widehat{\mu}_{i}(\overline{E}) - \widehat{\mu}_{d-i+1}(\overline{E}^\vee)  + 2\alpha(d) = 2\alpha(d).\]
It only remains to prove that 
\begin{equation}\label{ineqtransf0}
\zeta_i(\overline{E}) + \zeta_{d-i+1}(\overline{E}^\vee) \leq \alpha_{\mathrm{s}}(i) + \alpha_{\mathrm{s}}(d-i+1).
\end{equation}
For every $j \in\{0, \ldots d\}$, we consider the real number \[\sigma_j(\overline{E}) =-\sup \{\widehat{\deg}(\overline{F}) \ | \ F \subset E_{\overline{K}}, \ \dim F = j\}.\]
Then $\sigma_1(\overline{E}) = \zeta_1(\overline{E})$ and
$\sigma_j(\overline{E}) = \sigma_{d-j}(\overline{E}^\vee) - \widehat{\deg}(\overline{E})$ (see \cite[Theorem 1.1]{RoyThunder}). Let $\varepsilon > 0$ be a real number and let $G \subset E_{\overline{K}}$ be a vector subspace of dimension $i$ such that $-\widehat{\deg}(\overline{G}) \leq \sigma_i(\overline{E}) + \varepsilon$. We have 
\[\zeta_i(\overline{E})+\sigma_{i-1}(\overline{E}) \leq \zeta_i(\overline{G}) + \sigma_{i-1}(\overline{G})=\zeta_i(\overline{G}) + \zeta_{1}(\overline{G}^\vee)-\widehat{\deg}(\overline{G}).\]
Moreover,  Theorem \ref{thmminslopesintro} implies 
\begin{equation*}
\zeta_1(\overline{G}^\vee) + \zeta_{i}(\overline{G}) = \lim_{n \rightarrow +\infty}\frac{1}{n}(\pmin(S^n\overline{G}) + \pmax(S^n (\overline{G}^\vee))) = \alpha_s(\overline{G}) \leq \alpha_s(i),
\end{equation*}
so that 
\begin{equation*}
\begin{split}
\zeta_i(\overline{E})+\sigma_{i-1}(\overline{E}) \leq \alpha_s(i)-\widehat{\deg}(\overline{G}) &\leq \alpha_s(i)+\sigma_i(\overline{E}) + \varepsilon\\
 & = \alpha_s(i) +\sigma_{d-i}(\overline{E}^\vee)-\deg(\overline{E})+\varepsilon.
\end{split}
\end{equation*}
Similarly, we have
\begin{equation*}
\begin{split}
\zeta_{d+1-i}(\overline{E}^\vee)+\sigma_{d-i}(\overline{E}^\vee) &\leq \alpha_s(d-i+1) + \sigma_{i-1}(\overline{E}) -\deg(\overline{E}^\vee)+ \varepsilon\\
& = \alpha_s(d-i+1) + \sigma_{i-1}(\overline{E}) +\deg(\overline{E})+ \varepsilon.
\end{split}
\end{equation*}
Summing up, we find $\zeta_i(\overline{E}) + \zeta_{d+1-i}(\overline{E}^\vee) \leq \alpha_s(i) + \alpha_s(d-i+1) +2\varepsilon$, and we conclude by letting $\varepsilon$ tend to zero.

\end{proof}

\subsection{The function field case}\label{paragcdf} We assume now that $K=k(C_K)$ is the function field of a regular projective integral curve over a field $k$ and we let $\overline{E}$ be a hermitian $K$-vector space of dimension $d\geq1$. Consider the vector bundle $\mathcal{E}$ on $C_K$ defined by
\[\mathcal{E}(U) = \{x \in E \ | \ \|x\|_v \leq 1 \ \forall v \in U\}\]
for any non-empty open subset $U$ of $C_K$. By construction, we have 
\[\pmax(\overline{E}) = \sup_{\varphi \colon C' \rightarrow C_K} \sup\{\deg (\mathcal{F})/\rk(\mathcal{F}) \ | \ 0\ne \mathcal{F} \subset \varphi^* \mathcal{E} \text{ subbundle}\},\]
where $\varphi$ runs over all finite surjective morphisms of regular projective integral curves $\varphi \colon C' \rightarrow C_K$ on $k$. It follows from \cite[Theorem 7.2]{Moriwaki98} that $\pmax(S^n\overline{E}) = n\pmax(\overline{E})$ for any $n\geq 1$. Arguing as in paragraph \ref{sectionproofminslopes}, we find that $\zeta_{d}(\overline{E}) = \pmax(\overline{E}^\vee) = -\widehat{\mu}_d(\overline{E})$ (one can replace \cite[Theorem 1.1]{Chencompare} by \cite[Theorem 5.20]{Gaudronpentes}, which holds for function fields in arbitrary characteristic). By Proposition  \ref{propcomp0}, we have $\zeta_i(\overline{E})=-\widehat{\mu}_i(\overline{E})$ for all $1\leq i  \leq d$.

\bibliographystyle{plain}
\bibliography{EssminNF.bib}

\begin{thebibliography}{10}

\bibitem{BakerRumely}
Matthew~H. Baker and Robert Rumely.
\newblock Equidistribution of small points, rational dynamics, and potential
  theory.
\newblock {\em Ann. Inst. Fourier (Grenoble)}, 56(3):625--688, 2006.

\bibitem{Banaszczyk}
W.~Banaszczyk.
\newblock New bounds in some transference theorems in the geometry of numbers.
\newblock {\em Math. Ann.}, 296(4):625--635, 1993.

\bibitem{BermanBoucksom}
Robert Berman and S\'{e}bastien Boucksom.
\newblock Growth of balls of holomorphic sections and energy at equilibrium.
\newblock {\em Invent. Math.}, 181(2):337--394, 2010.

\bibitem{BombieriVaaler}
E.~Bombieri and J.~Vaaler.
\newblock On {S}iegel's lemma.
\newblock {\em Invent. Math.}, 73(1):11--32, 1983.

\bibitem{BosserGaudron}
Vincent Bosser and \'{E}ric Gaudron.
\newblock Logarithmes des points rationnels des vari\'{e}t\'{e}s
  ab\'{e}liennes.
\newblock {\em Canad. J. Math.}, 71(2):247--298, 2019.

\bibitem{BoucksomChen}
S\'{e}bastien Boucksom and Huayi Chen.
\newblock Okounkov bodies of filtered linear series.
\newblock {\em Compos. Math.}, 147(4):1205--1229, 2011.

\bibitem{BDPP}
S\'{e}bastien Boucksom, Jean-Pierre Demailly, Mihai P\u{a}un, and Thomas
  Peternell.
\newblock The pseudo-effective cone of a compact {K}\"{a}hler manifold and
  varieties of negative {K}odaira dimension.
\newblock {\em J. Algebraic Geom.}, 22(2):201--248, 2013.

\bibitem{BMPS}
Jos\'{e}~Ignacio Burgos~Gil, Atsushi Moriwaki, Patrice Philippon, and
  Mart\'{\i}n Sombra.
\newblock Arithmetic positivity on toric varieties.
\newblock {\em J. Algebraic Geom.}, 25(2):201--272, 2016.

\bibitem{BPRS}
Jos\'{e}~Ignacio Burgos~Gil, Patrice Philippon, Juan Rivera-Letelier, and
  Mart\'{\i}n Sombra.
\newblock The distribution of {G}alois orbits of points of small height in
  toric varieties.
\newblock {\em Amer. J. Math.}, 141(2):309--381, 2019.

\bibitem{BPS14}
Jos\'{e}~Ignacio Burgos~Gil, Patrice Philippon, and Mart\'{\i}n Sombra.
\newblock Arithmetic geometry of toric varieties. {M}etrics, measures and
  heights.
\newblock {\em Ast\'{e}risque}, (360):vi+222, 2014.

\bibitem{BPSmin}
Jos\'{e}~Ignacio Burgos~Gil, Patrice Philippon, and Mart\'{\i}n Sombra.
\newblock Successive minima of toric height functions.
\newblock {\em Ann. Inst. Fourier (Grenoble)}, 65(5):2145--2197, 2015.

\bibitem{ChL}
Antoine Chambert-Loir.
\newblock Mesures et \'{e}quidistribution sur les espaces de {B}erkovich.
\newblock {\em J. Reine Angew. Math.}, 595:215--235, 2006.

\bibitem{Charles}
Fran\c{c}ois Charles.
\newblock Arithmetic ampleness and an arithmetic {B}ertini theorem.
\newblock Preprint, 2019.

\bibitem{Chenfujita}
Huayi Chen.
\newblock Arithmetic {F}ujita approximation.
\newblock {\em Ann. Sci. \'{E}c. Norm. Sup\'{e}r. (4)}, 43(4):555--578, 2010.

\bibitem{Chenpolygone}
Huayi Chen.
\newblock Convergence des polygones de {H}arder-{N}arasimhan.
\newblock {\em M\'{e}m. Soc. Math. Fr. (N.S.)}, (120):116, 2010.

\bibitem{Chendiff}
Huayi Chen.
\newblock Differentiability of the arithmetic volume function.
\newblock {\em J. Lond. Math. Soc. (2)}, 84(2):365--384, 2011.

\bibitem{ChenMathZ}
Huayi Chen.
\newblock Majorations explicites des fonctions de {H}ilbert-{S}amuel
  g\'{e}om\'{e}trique et arithm\'{e}tique.
\newblock {\em Math. Z.}, 279(1-2):99--137, 2015.

\bibitem{Chencompare}
Huayi Chen.
\newblock Sur la comparaison entre les minima et les pentes.
\newblock In {\em Publications math\'{e}matiques de {B}esan\c{c}on. {A}lg\`ebre
  et th\'{e}orie des nombres. 2018}, volume 2018 of {\em Publ. Math.
  Besan\c{c}on Alg\`ebre Th\'{e}orie Nr.}, pages 5--23. Presses Univ.
  Franche-Comt\'{e}, Besan\c{c}on, 2018.

\bibitem{Cheniccm}
Huayi Chen.
\newblock Comparison of some invariants of {E}uclidean lattices.
\newblock Proceedings of ICCM, 2019.

\bibitem{ChenMoriwaki}
Huayi Chen and Atsushi Moriwaki.
\newblock Algebraic dynamical systems and {D}irichlet's unit theorem on
  arithmetic varieties.
\newblock {\em Int. Math. Res. Not. IMRN}, (24):13669--13716, 2015.

\bibitem{CMadeliccurves}
Huayi Chen and Atsushi Moriwaki.
\newblock {\em Arakelov Geometry over Adelic Curves}, volume 2258 of {\em
  Lecture Notes in Mathematics}.
\newblock Springer, 2019.

\bibitem{Cutkovsky}
Steven~Dale Cutkosky.
\newblock Teissier's problem on inequalities of nef divisors.
\newblock {\em J. Algebra Appl.}, 14(9):1540002, 37, 2015.

\bibitem{Das}
Omprokash Das.
\newblock Finiteness of log minimal models and nef curves on $3$ -folds in
  characteristic $p>5$.
\newblock {\em Nagoya Mathematical Journal}, page 1–34.

\bibitem{DeTemple}
Duane~W. Detemple.
\newblock {T}he non-integer property of sums of reciprocals of successive
  integers.
\newblock {\em The Mathematical Gazette}, 75(472):193--194, 1991.

\bibitem{FavreRivera}
Charles Favre and Juan Rivera-Letelier.
\newblock Equidistribution quantitative des points de petite hauteur sur la
  droite projective.
\newblock {\em Math. Ann.}, 335:311–361, 2006.

\bibitem{Gaudronpentes}
\'{E}ric Gaudron.
\newblock Pentes des fibr\'{e}s vectoriels ad\'{e}liques sur un corps global.
\newblock {\em Rend. Semin. Mat. Univ. Padova}, 119:21--95, 2008.

\bibitem{Gaudronart18}
{\'E}ric Gaudron.
\newblock Minima and slopes of rigid adelic spaces.
\newblock In Ga{\"e}l R{\'e}mond and Emmanuel Peyre, editors, {\em {Arakelov
  Geometry and {D}iophantine applications}}. 2020.
\newblock 44 pages, available at https://hal.archives-ouvertes.fr/hal-02445064.

\bibitem{GR13}
\'{E}ric Gaudron and Ga\"{e}l R\'{e}mond.
\newblock Minima, pentes et alg\`ebre tensorielle.
\newblock {\em Israel J. Math.}, 195(2):565--591, 2013.

\bibitem{GRisogenies}
\'{E}ric Gaudron and Ga\"{e}l R\'{e}mond.
\newblock Polarisations et isog\'{e}nies.
\newblock {\em Duke Math. J.}, 163(11):2057--2108, 2014.

\bibitem{GRSiegel}
\'{E}ric Gaudron and Ga\"{e}l R\'{e}mond.
\newblock Corps de {S}iegel.
\newblock {\em J. Reine Angew. Math.}, 726:187--247, 2017.

\bibitem{GublerBogomolov}
Walter Gubler.
\newblock The {B}ogomolov conjecture for totally degenerate abelian varieties.
\newblock {\em Invent. Math.}, 169(2):377--400, 2007.

\bibitem{Gublerequi}
Walter Gubler.
\newblock Equidistribution over function fields.
\newblock {\em Manuscripta Math.}, 127(4):485--510, 2008.

\bibitem{IkomaBertini}
Hideaki Ikoma.
\newblock A {B}ertini-type theorem for free arithmetic linear series.
\newblock {\em Kyoto J. Math.}, 55(3):531--541, 2015.

\bibitem{Ikoma15}
Hideaki Ikoma.
\newblock On the concavity of the arithmetic volumes.
\newblock {\em Int. Math. Res. Not. IMRN}, (16):7063--7109, 2015.

\bibitem{KavehKhovanskii}
Kiumars Kaveh and A.~G. Khovanskii.
\newblock Newton-{O}kounkov bodies, semigroups of integral points, graded
  algebras and intersection theory.
\newblock {\em Ann. of Math. (2)}, 176(2):925--978, 2012.

\bibitem{Laz}
Robert Lazarsfeld.
\newblock {\em Positivity in algebraic geometry. {I}}, volume~48 of {\em
  Ergebnisse der Mathematik und ihrer Grenzgebiete. 3. Folge. A Series of
  Modern Surveys in Mathematics [Results in Mathematics and Related Areas. 3rd
  Series. A Series of Modern Surveys in Mathematics]}.
\newblock Springer-Verlag, Berlin, 2004.
\newblock Classical setting: line bundles and linear series.

\bibitem{LazII}
Robert Lazarsfeld.
\newblock {\em Positivity in algebraic geometry. {II}}, volume~49 of {\em
  Ergebnisse der Mathematik und ihrer Grenzgebiete. 3. Folge. A Series of
  Modern Surveys in Mathematics [Results in Mathematics and Related Areas. 3rd
  Series. A Series of Modern Surveys in Mathematics]}.
\newblock Springer-Verlag, Berlin, 2004.
\newblock Positivity for vector bundles, and multiplier ideals.

\bibitem{LazMus}
Robert Lazarsfeld and Mircea Musta\c{t}\u{a}.
\newblock Convex bodies associated to linear series.
\newblock {\em Ann. Sci. \'{E}c. Norm. Sup\'{e}r. (4)}, 42(5):783--835, 2009.

\bibitem{Moriwaki94}
Atsushi Moriwaki.
\newblock Arithmetic {B}ogomolov-{G}ieseker's inequality.
\newblock {\em Amer. J. Math.}, 117(5):1325--1347, 1995.

\bibitem{Moriwaki98}
Atsushi Moriwaki.
\newblock Relative {B}ogomolov's inequality and the cone of positive divisors
  on the moduli space of stable curves.
\newblock {\em J. Amer. Math. Soc.}, 11(3):569--600, 1998.

\bibitem{MoriwakiMZ2012}
Atsushi Moriwaki.
\newblock Arithmetic linear series with base conditions.
\newblock {\em Math. Z.}, 272(3-4):1383--1401, 2012.

\bibitem{MoriwakiZariski}
Atsushi Moriwaki.
\newblock Zariski decompositions on arithmetic surfaces.
\newblock {\em Publ. Res. Inst. Math. Sci.}, 48(4):799--898, 2012.

\bibitem{MoriwakiMAMS}
Atsushi Moriwaki.
\newblock Adelic divisors on arithmetic varieties.
\newblock {\em Mem. Amer. Math. Soc.}, 242(1144):v+122, 2016.

\bibitem{Pekker}
Alexander Pekker.
\newblock On successive minima and the absolute {S}iegel's lemma.
\newblock {\em J. Number Theory}, 128(3):564--575, 2008.

\bibitem{RoyThunder}
Damien Roy and Jeffrey~Lin Thunder.
\newblock An absolute {S}iegel's lemma.
\newblock {\em J. Reine Angew. Math.}, 476:1--26, 1996.

\bibitem{SUZ}
L.~Szpiro, E.~Ullmo, and S.~Zhang.
\newblock \'{E}quir\'{e}partition des petits points.
\newblock {\em Invent. Math.}, 127(2):337--347, 1997.

\bibitem{Yuanbig}
Xinyi Yuan.
\newblock Big line bundles over arithmetic varieties.
\newblock {\em Invent. Math.}, 173(3):603--649, 2008.

\bibitem{Zhangplav}
Shouwu Zhang.
\newblock Positive line bundles on arithmetic varieties.
\newblock {\em J. Amer. Math. Soc.}, 8(1):187--221, 1995.

\bibitem{Zhangadelic}
Shouwu Zhang.
\newblock Small points and adelic metrics.
\newblock {\em J. Algebraic Geom.}, 4(2):281--300, 1995.

\end{thebibliography}

\end{document}